\documentclass[10pt,a4paper]{amsart}
\usepackage[utf8]{inputenc}
\usepackage[english]{babel}
\usepackage{amsthm}

\usepackage{amsfonts,amssymb,amsmath}
\usepackage[color=green!30]{todonotes}

\usepackage{enumerate}

\usepackage{xcolor}
\usepackage[colorlinks,
    linkcolor={red!50!black},
    citecolor={blue!50!black},
    urlcolor={blue!80!black}]{hyperref}
\usepackage{tikz}
\usetikzlibrary{matrix}
\usepackage[all]{xy}
\usetikzlibrary{calc, matrix, arrows, cd}
\tikzset{node distance=2cm, auto}

\newcommand{\h}[1]{\hat{#1}}
\newcommand{\ti}[1]{\tilde{#1}}

\newcommand{\mK}{\mathbb{K}}
\newcommand{\mZ}{\mathbb{Z}}
\newcommand{\mS}{\mathbb{S}}
\newcommand{\mC}{\mathbb{C}}
\newcommand{\mP}{\mathbb{P}}

\newcommand{\ph}{\varphi}

\newcommand{\cO}{\mathcal{O}}

\newcommand{\bm}{\begin{pmatrix}}
\newcommand{\ema}{\end{pmatrix}}
\newcommand{\bsm}{\left(\begin{smallmatrix}}
\newcommand{\esm}{\end{smallmatrix}\right)}
\newcommand{\tor}{\operatorname{tor}}

\newcommand{\Hom}{\operatorname{Hom}}

\newcommand{\Id}{\operatorname{Id}}

\newcommand{\Tr}{\operatorname{Tr}}

\newcommand{\Ext}{\operatorname{Ext}}

\newcommand{\Sldeux}{\operatorname{SL}_2}
\newcommand{\SL}{\operatorname{SL}}
\newcommand{\GL}{\operatorname{GL}}
\newcommand{\Sltrois}{\operatorname{SL}_3}

\newcommand{\Gldeux}{\operatorname{GL}_2}

\newcommand{\len}{\operatorname{length}}
\newcommand{\al}{\alpha}
\newcommand{\be}{\beta}
\newcommand{\g}{\gamma}
\newcommand{\de}{\delta}

\newcommand{\la}{\lambda}

\newcommand{\Si}{\Sigma}
\newcommand{\frp}{\mathfrak{p}}

\newcommand{\benu}{\begin{enumerate}}
\newcommand{\eenu}{\end{enumerate}}
\newcommand{\brho}{{\bar{\rho}}}

\tikzset{node distance=2cm, auto}
\theoremstyle{definition}
\newtheorem{theo}{Theorem}[section]
\newtheorem{lem}[theo]{Lemma}
\newtheorem{cor}[theo]{Corollary}
\newtheorem{prop}[theo]{Proposition}
\theoremstyle{remark}
\newtheorem{defi}[theo]{Definition}
\newtheorem{remk}[theo]{Remark}
\newtheorem{ex}[theo]{Example}
\newtheorem{question}[theo]{Question}
\newtheorem*{nota}{Notation}
\newtheorem*{theo*}{Theorem}


\begin{document}

\title{Torsion function on character varieties}


\begin{abstract}
In this paper we define the Reidemeister torsion as a rational function on the geometric components of the character variety of a one-cusped hyperbolic manifold $M$. We study its poles and zeros, and we deduce sufficient conditions on the manifold $M$ for this function being non-constant. 
\end{abstract}
\author{Leo Benard}
\address{Georg--August Universit\"at, Mathematisches Institut, Busenstrasse 3-5, 37073 G\"ottingen}
\email{leo.benard@mathematik.uni-goettingen.de}
\date{}
\maketitle

\section{Introduction}
The Reidemeister torsion has been introduced as a combinatorial invariant of homological complexes in 1935 by both Reidemeister (in \cite{Rei}) and Franz (in \cite{Fra}) independently. It also appears in a more algebraic context in the seminal work of Cayley (see \cite[Appendix B]{GKZ94}) in 1848. Later on, it has been shown by Chapman (\cite{Chapman, Cohen}) to be a topological invariant of manifolds.
\medbreak

In this article $M$ will be a 3-manifold, whose boundary $\partial M$ is a torus, with rational homology of a circle, e.g. the exterior of a knot in a rational homology sphere. Given a representation $\rho \colon  \pi_1(M) \to \Sldeux(\mC)$ such that the  complex $C^*(M, \rho)$ of $\rho$-twisted cohomology with coefficients in $\mC^2$ is acyclic, the torsion $\tor(M,\rho)$ of this complex is a complex-valued invariant of the pair $(M,\rho)$, defined up to sign. Since for any representation $\rho' \colon  \pi_1(M) \to \Sldeux(\mC)$ conjugate to $\rho$, the invariants $\tor(M,\rho)$ and $\tor(M, \rho')$ do coincide, it is natural to define the Reidemeister torsion as a rational function on the algebraic space of conjugacy classes of such representations, namely the character variety. It is done rigorously in this article. 
\medbreak

More precisely, let $X$ be a one-dimensional component of the character variety. The torsion function is seen as a non-zero element of the function field $k(X)$ of $X$. While it is not usually defined in the way it is in this paper, the torsion function has been long established as such. The question of how to compute this function and whether it has some zeros or not is still under investigation. It is known to be a constant function on the character varieties of torus knots (the case of the trefoil knot is explicitly computed in Section \ref{section examples}). The first non-constant computation was done by Teruaki Kitano for the figure-eight knot complement in \cite{Kit94}. Since then, because of its proximity with the twisted Alexander polynomial, there has been many more studies of this torsion. In \cite{DFJ}, the authors address several questions on the twisted Alexander polynomial. In particular they conjecture that the degree of this polynomial is related in a very precise way to the genus of the knot. The torsion function turns out to be the specialization of this polynomial at $t=1$.

\medbreak

When the interior of $M$ carries a complete hyperbolic structure, it comes with a distinguished representation $\rho \colon \pi_1(M) \to \operatorname{PSL}_2(\mC)$ which always lifts to $\SL_2(\mC)$. It is called a holonomy representation. A component $X$ of the character variety $X(M)$ that contains the character of a holonomy representation is called a geometric component. It follows from the work of Thurston that such a component has dimension equal to the number of cusps of $M$, see \cite[Section 4.5]{Sha02} for an overview of Thurston's proof.

\medbreak

The first result of this article is the following:
\begin{theo}  \label{acyclic finite}
Let $M$ will be a hyperbolic 3-manifold whose boundary $\partial M$ is a torus. Let $X$ be a geometric component of the character variety $X(M)$ fo $M$. Then the torsion defines a regular function $\tor(M)$ on $X$. It vanishes at a character $\chi$ of $X$ if and only if the vector space $H^1(M, \brho)$ is non-trivial, where $\brho$ is a representation $\brho \colon  \pi_1(M) \to \Sldeux(\mC)$ whose character is $\chi$.
\end{theo}
The fact that the torsion has no poles on $X$ was expected. Yet, to our knowledge, there has been no such formal statement to date. The characterization in terms of jump of dimension of the vector space $H^1(M, \brho)$ is  relates the torsion with the deformation theory of semi-simple representations in $\Sltrois(\mC)$. More precisely, given an irreducible representation $\rho \colon   \pi_1(M) \to \Sldeux(\mC)$, one can construct a semi-simple representation $\ti{\rho} \colon  \pi_1(M) \to \Sltrois(\mC)$ by defining $\ti{\rho} = \bsm \rho & 0 \\ 0 & 1 \esm$. A classical dimensional argument (see \cite[Section 5]{HP15} for instance), shows that if $\ti{\rho}$ is deformable into irreducible representations in the character variety $X(M, \Sltrois(\mC))$ then there exists necessarily a reducible, non-semi-simple representation $\ti{\rho}'$ with the same character, namely $\ti{\rho}' = \bsm \rho & z \\ 0 & 1 \esm$ where $z \colon  \pi_1(M) \to \mC^2$ represents a non-trivial class in $H^1(M, \rho)$. More generally, given $\la \in \mC^*$ and a surjective abelianization map $\ph \colon  \pi_1(M) \to \mZ$, it is proven in \cite{HP15} that if the representation $\ti{\rho}_\la = \bsm \la^\ph \rho & 0 \\ 0 & \la^{-2\ph} \esm$ is deformable into irreducible representations, then the twisted Alexander polynomial $\Delta_\rho(\la^3)$ vanishes.
 A converse statement is proved in the case when $\la^3$ is a simple root of the twisted Alexander polynomial. Our Theorem \ref{acyclic finite} corresponds to the case where $\la=1$, saying that if $\Delta_\rho(1) = 0$ then there exists a non-trivial $z\in H^1(M, \rho)$, hence a non-semi-simple $\ti{\rho}'$ as above. It generalizes a basic fact from the $\Sldeux(\mC)$-case.
\medbreak

The second part of this article focuses on the asymptotic behavior of the torsion function on the variety $X$. There is a canonical way to define a compact Riemann surface $\hat{X}$ birational to $X$, by desingularizing $X$ and adding points at infinity, see Section \ref{valu}. Those points added at infinity will be called ideal points in $\h{X}$. In this paper we extend the torsion to a rational function on $\h{X}$. Given $x$ an ideal point in $\h{X}$, we describe in Section \ref{tree} a construction due to Marc Culler and Peter Shalen in the early 80's. This construction produces an incompressible embedded surface $\Si \subset M$ associated to $x$. 

Although no representation of the fundamental group of $M$ corresponds to this ideal point, it corresponds to a representation of the fundamental group of the surface $\Sigma$, denoted by $\brho_\Si\colon \pi_1(\Si) \to \Sldeux(\mC)$. It is called the \textit{residual representation}, and we say that this representation is non-trivial if there exists $\gamma$ in $\pi_1(\Sigma)$ such that $\Tr \brho_\Si(\gamma) \neq 2$. The residual representation  $\brho_\Si$ representation is known (see~\cite{CCGLS}) to map the class of the boundary curve $\partial \Si$ on a matrix whose eigenvalues are roots of unity. Moreover, the order of those roots of unity divides the minimal number of boundary components of any connected component of the surface $\Si$. Now we can state the second result of this article:

\begin{theo} \label{acyclic ideal}
Let $x \in \h{X}$ be an ideal point in the smooth projective model of a geometric component $X$ of the character variety. Assume that the associated incompressible surface $\Sigma$ is a union of parallel homeomorphic copies $\Si_i$ such that $M \setminus \Si_i$ is a (union of) handlebodie(s). If the residual representation $\brho_\Sigma$ is non-trivial, and $\Tr \brho_\Sigma([\partial \Sigma]) =2$, then the torsion function $\tor(M)$ has a pole at $x$. 
\end{theo}
A surface whose complement in a manifold M is a (union of) handlebodies is called a free surface.
We deduce the following corollary:
\begin{cor}
\label{cor:main}
Let $M$ be a hyperbolic manifold and $X$ be a geometric component of its $\Sldeux(\mC)$ character variety.
Assume that an ideal point of the smooth projective model $\hat{X}$ of $X$ detects an incompressible surface $\Sigma$ which is union of parallel free copies, such that the residual representation $\brho_\Sigma$ is non-trivial, and that $\Tr \brho_\Sigma([\partial \Sigma]) =2$. Then the torsion function is not constant on the component $X$.
\end{cor}

Remark that for sake of brevity we state both Theorems \ref{acyclic finite} and \ref{acyclic ideal} for $X$ a geometric component in $X(M)$, but along the paper it will be enough to assume that $X$ is a one-dimensional component that contains the character of an irreducible representation, that the complex $C^*(M, \rho)$ is acyclic for some representation $\rho$ whose character lies in $X$ and that the detected surface $\Sigma$ has non-positive Euler characteristic. 

For instance, any component $X$ of $X(M)$ is also one-dimensional when the manifold $M$ is small (a small manifold is a manifold that contains no \textit{closed} incompressible surfaces. Two-bridge knot complements are examples of small manifolds.). 

The assumption that the eigenvalues of $\brho_\Sigma([\partial \Sigma])$ are equal to 1 is motivated by numerical computations and by the fact that it occurs in many known examples, for instance if the associated surface is non-separating in $M$ (including Seifert surfaces). In fact it has been a difficult task to find roots of unity different of $\pm 1$ with this construction, see \cite{Dun98} for the first known example. On the other hand, in \cite{Che05} Chesebro proves that any root of unity can be obtained as the eigenvalue of the boundary curve of an incompressible surface detected by an ideal point. 

Finally, the assumption that the surface is free automatically satisfied whenever $M$ is a small manifold.

We provide infinite families of candidates for which the corollary may apply, that are knots in $\mS^3$. Unfortunately, (except for the case of the figure-eight knot where we checked it by hand), the non-triviality condition, despite generic in the character variety, seems difficult to check in general.
\begin{cor}
\label{cor:Examples}
If $M$ is the exterior of a twist knot $J(2,2n)$ or $J(2,2n+1)$ different from the trefoil knot, or the exterior of a "double twist knot" $J(3,2n)$ with negative $n$ (see Figure \ref{fig:triple}), if the residual representation $\brho_\Sigma$ given by an ideal point is non-trivial, then the torsion function is non-constant on the geometric component of the character variety.
\end{cor}

\begin{figure}[h]
\begin{center}
\def\svgwidth{0.8\columnwidth}
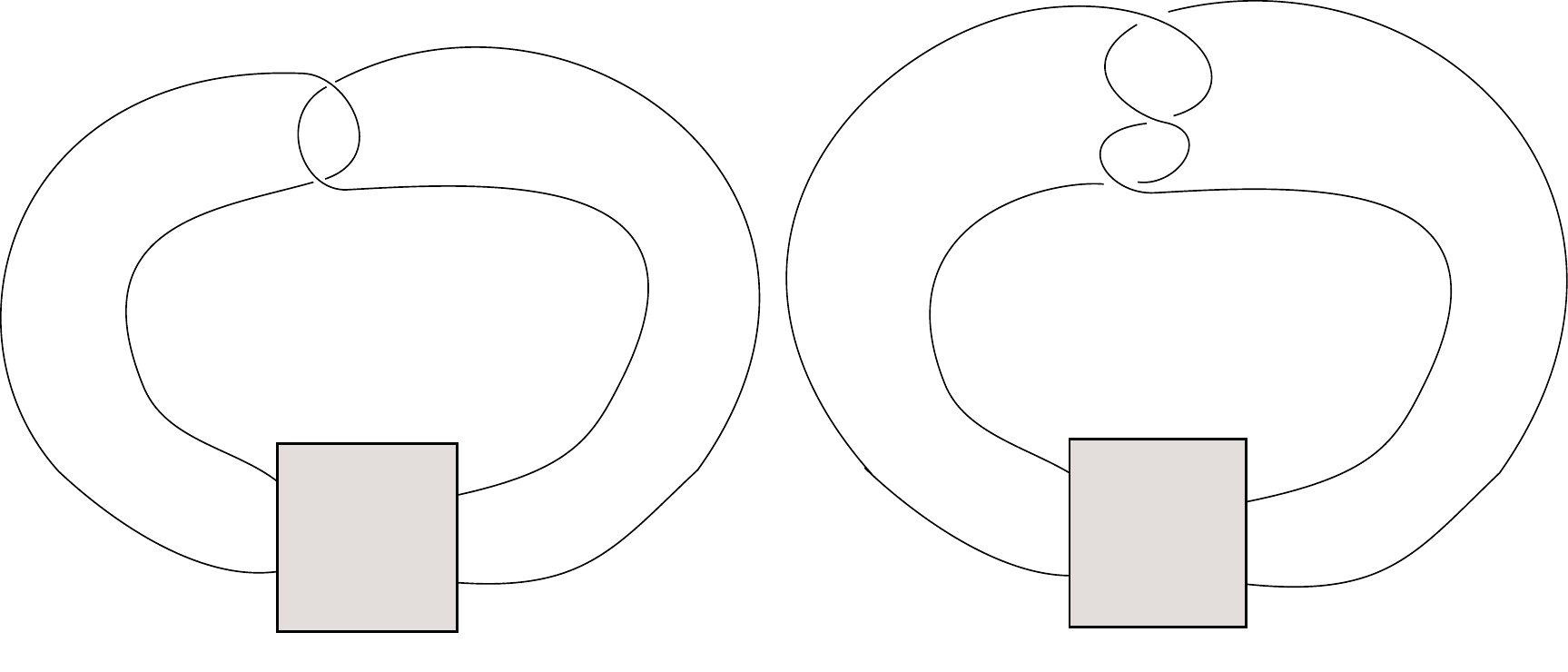
\caption{\label{fig:triple} A diagram of $J(2,2n)$ (respectively $J(2,2n+1)$) on the left, and of $J(3,2n)$ on the right.}
\end{center}
\end{figure}
\begin{proof}
To begin with, the incompressible surfaces obtained by the construction of Culler--Shalen for two-bridge knots complements have been fully classified by Tomotada Ohtsuki in \cite{Ohtsuki}. This result relies on the classification of incompressible surfaces in two-bridges knot complements given by Allen Hatcher and William Thurston in \cite{HT85}. In particular, \cite[Theorem 5.4]{Ohtsuki} implies that those surfaces are necessarily connected or union of parallel connected copies.

Now any twist or double twist knot is known to be hyperbolic, the trefoil knot excepted. Moreover, the geometric component is the unique component (of irreducible type) of their character variety (see \cite{MPvL11}). Hence we need to argue that an incompressible surface satisfying the hypothesis of Corollary \ref{cor:main} is detected. For the figure-eight knot, it is computed explicitly in Remark \ref{eight knot}. For any other $M$, it follows from the computations of the $A$-polynomials (\cite{BMZ97} or  \cite[Figures 3 and 4]{HS04}) that a Seifert surface is detected by an ideal point, since a boundary slope of the $A$-polynomial is a boundary slope of a detected incompressible surface (\cite[Theorem 3.4]{CCGLS}). Since those manifolds are small manifolds, this Seifert surface is free, and the corollary follows.
\end{proof}
\medbreak

Since we expect the torsion to be non-constant on the geometric component of any small hyperbolic three manifold, we ask the following question, compare with the end of \cite[Subsection 8.4]{DFJ}:
\begin{question}
Is it true that a geometric component of a small hyperbolic three manifold detects necessarily an incompressible surface whose boundary curve has eigenvalue 1?
\end{question}
Two-bridge knots are known to contain non-Seifert incompressible surfaces with only 2 boundary components, in this case the eigenvalue of the boundary curve is $\pm1$. Yet, as mentioned above, the boundary curve of any Seifert surface has eigenvalue equal to 1 through the residual representation. The figure-eight knot provides an example of a non-Seifert detected surface with eigenvalue equal to 1. 
\begin{remk}
The character variety of torus knots is known to be a union of irreducible curves. Each of the surfaces detected by an ideal point in a component of irreducible type is an incompressible annulus. In particular it has an infinite cyclic fundamental group, so that the residual representation is either trivial, either the boundary curve is non-trivial. Hence Corollary \ref{cor:main} does not apply; as mentioned before, the torsion function is known to be constant in this situation.
\end{remk}
\begin{nota}
In this article we are interested in the following situation: $M$ is a 3-manifold, compact, connected and orientable, with toroidal boundary $\partial M = \mS^1 \times \mS^1$ and with rational homology of a circle. We fix $k$ an algebraically closed field of characteristic 0. We study twisted cohomology groups given by some representations $\rho \colon  \pi_1(M) \to \Sldeux(K)$, where $K$ is the function field of some variety defined over $k$. Given $v \colon K^* \to \mZ$ a discrete $k$-valuation, we will denote by $\cO_v \subset K$ the valuation subring of $K$, with residual field $k$.  We use the following notations:
\benu
\item
We denote by  $H^*(M, \rho)$ the twisted cohomology $K$-vector spaces, with $\pi_1(M)$ acting on $K^2$ through $\rho$.
\item
Whenever the representation $\rho \colon  \pi_1(M) \to \Sldeux(\cO_v)$ has image in $\Sldeux(\cO_v)$, we denote by $H^*(M, \rho)_v$ the twisted $\cO_v$-modules with coefficients in $\cO_v^2$.
\item
In this case, we denote by $\brho \colon  \pi_1(M) \to \Sldeux(k)$ the composition of $\rho$ with the residual map $\cO_v \to k$, and by 
 $H^*(M, \brho)$ the twisted $k$-vector spaces for the action of $\pi_1(M)$ on $k^2$.
\item
Given a homomorphism $\la \colon  \pi_1(M) \to  k^\ast$, we denote by $H^\ast(M,\la)$ the twisted cohomology groups with action of $\la$ by multiplication on the field $k$. Of course, when $\la$ is constant equal to $1$, we keep the notation $H^\ast(M,k)$.
\eenu
\end{nota}

This paper is organized as follows,  in Section \ref{sec:1} we give the basics about character varieties that will be used along this work, in Section \ref{tree} we give an overview of the Culler-Shalen theory and present the tree-theoretical arguments that lead to the results of this article, in Section \ref{Reidem} we review some various definitions of the Reidemeister torsion, and  we define the torsion function. Then in Section \ref{sec:Finite} we prove Theorem \ref{acyclic finite} and compute a series of examples, and in Section \ref{sec:Ideal} we prove Theorem \ref{acyclic ideal}.

\subsubsection*{Acknowledgement}
This work is part of the author's PhD dissertation, which has been conducted in Sorbonne-Universit\'e, Institut de Math\'ematiques de Jussieu -- Paris Rive Gauche. The author has benefited of a constant support from his advisor Julien March\'e, and thanks him for his time and help. He thanks Teruaki Kitano and Joan Porti for having suggested to study the acyclic torsion, and for sharing with him helpful advices on the topic. Finally, we thank the anonymous referees for their numerous remarks and comments that have contributed to improve the writing of the paper. In particular, comments of an anonymous referee on a previous version of this manuscript have led to a complete rewriting of Lemma \ref{lem:NonSplit} and of Subsection \ref{non-split case}. We thank him/her for that.

\section{Character varieties}
\label{sec:1}
In this section we give definitions relative to character varieties of the fundamental group of a 3-manifold $M$ (Subsection \ref{subs:character}) and we define the tautological representation (Subsection \ref{taut}).  We end this section with examples in Subsection \ref{subs:Examples}. A more detailed treatment of what follows can be found in the first section of \cite{Ben16}.
\medbreak

\subsection{Character varieties and irreducible characters}
\label{subs:character}
In this subsection we define the character variety of a 3-manifold, and discuss the notion of irreducibility for characters.
\medbreak
\begin{defi}
Let $M$ be a 3-manifold, and $\pi_1(M)$ its fundamental group. We define the character variety of $M$ as the algebraic set of conjugacy classes of representations $\rho \in \Hom(\pi_1(M), \Sldeux(k))$, namely it is the algebro-geometric quotient $X(M) = \Hom(\pi_1(M),\Sldeux(k))// \Sldeux(k)$. This quotient can be described as the set of equivalence classes of representations $\rho \colon \pi_1(M) \to \SL_2(k)$, for the equivalence relation given by
$$\rho \sim \rho'  \text{ if for all } \gamma \in \pi_1(M), \ \Tr \rho(\gamma) = \Tr \rho'(\gamma).$$
\end{defi}
The following theorem allows us to identify functions on the character variety with the so-called trace functions.

\begin{theo}[\cite{P87}]
The algebra $$B[M] = k[Y_\g, \g \in \pi_1(M)]/(Y_e-2,Y_\g Y_\de - Y_{\g\de}-Y_{\g\de^{-1}}, \g, \de \in \pi_1(M) )$$ is isomorphic to the function algebra $k[X(M)]$.
\end{theo}
\begin{remk}
The relation in the algebra $B[M]$ arises from the well-know trace relation in $\SL_2(\mC)$:
for any matrices $A$ and $B$ in $\SL_2(\mC)$, the following holds
\begin{align}
\label{eq:TraceRelation}
\Tr A \Tr B = \Tr (AB)+\Tr(AB^{-1})
\end{align}
\end{remk}
\begin{defi}
A $k$-\textit{character} is an algebra morphism $\chi \colon  B[M] \to k$ (it is a $k$-point of the character variety in the sense of algebraic geometry). Any representation $\rho \colon  \pi_1(M) \to \Sldeux(k)$ induces a $k$-character $\chi_\rho \colon  B[M] \to k$ that sends $Y_\g$ to $\Tr(\rho(\g))$. 
\end{defi}
\begin{remk}
This definition generalizes to $R$-characters for any $k$-algebra $R$. When $R=k$, we will frequently omit the $k$ and say simply a character.
\end{remk}

\begin{defi}
Let $\mK$ be a field extension of $k$. A representation $\rho\colon \pi_1(M)\to~\Sldeux(\mK)$ is \textit{reducible} if there exists a $\rho(\pi_1(M))$-invariant line in $\mK^2$,  else it is \textit{irreducible}. Moreover, such a representation will be said \textit{absolutely irreducible} if it is irreducible in an algebraic closure of $\mK$.
\end{defi}
The following standard lemma tells us that this notion can be defined at the level of characters.
For any elements $\al, \be \in \pi_1(M)$, we use the standard notation $[\al,\be]$.
 for the commutator $\al\be\al^{-1}\be^{-1}$.
\begin{lem}[\cite{CS83,Mar15}]
\label{lem:Commutators}
A representation $\rho\colon  \pi_1(M) \to~ \Sldeux(\mK)$ is absolutely irreducible if and only if there exists $\alpha, \beta$ in $\pi_1(M)$ such that $\Tr(\rho(\al\be\al^{-1}\be^{-1})) \neq 2$.
\end{lem}

\begin{defi}
For any $\al, \be  \in \pi_1(M)$, we denote by $\Delta_{\al,\be} \in B[M]$ the function $Y_\al^2+Y_\be^2+Y_{\al\be}^2-Y_\al Y_\be Y_{\al\be}-4=Y_{[\al,\be]}-2$. For any $k$-algebra $R$, we will say that an $R$-character $\chi$ is \textit{irreducible} if there exists $\al, \be \in \pi_1(M)$ such that $\chi(\Delta_{\al, \be}) \neq 0$. If not, we say that it is \textit{reducible}. A character $\chi$ will be said \textit{central} if $\chi(Y_\g)^2=4$ for any $\g \in \pi_1(M)$.
\end{defi}

\begin{defi}\label{irreducible type}
An irreducible component $X \subset X(M)$ that contains only reducible characters will be said \textit{of reducible type}. A component that contains an irreducible character (equivalently, a dense open subset of irreducible characters, since being reducible is a Zariski closed condition) will be said \textit{of irreducible type}. 
\end{defi}

 The following proposition will be of crucial use in the next section.

 \begin{prop}{\cite{S93}\cite[Proposition 3.4]{Mar15}} \label{repr}
Let $\mK$ be either the algebraically closed field $k$ or a transcendental extension of degree one of $k$. Then the $\mK$-irreducible characters of $X(M)$ correspond bijectively to $\Gldeux(\mK)$-conjugacy classes of absolutely irreducible representations $\rho\colon \pi_1(M)\to \Sldeux(\mK)$.
\end{prop}
\subsection{The tautological representation}
 \label{taut}
 In this subsection we define the so-called tautological representation. It has a long story in the study of character varieties, see among others \cite{CS83,CGLS,CCGLS,DFJ}, and for instance \cite{FKN} for character varieties in higher rank groups. It will be our main tool to define the Reidemeister torsion globally on the character variety.
\medbreak
Let $X$ be an irreducible component of $X(M)$ of irreducible type (in the sense of Definition \ref{irreducible type}).
The component $X$ corresponds to a minimal prime ideal $\frp$ of $B[M]$ such that $k[X] =B[M]/\frp$ is the algebra of functions of $X$. Denote by $k(X)$ the fraction field of $k[X]$, it is called the function field of $X$. Let $\chi_X$ the composition $\chi_X \colon B[M] \to k[X] \to k(X)$. Since $X$ is of irreducible type, $\chi_X$ induces an irreducible $k(X)$-character.
The following is an immediate consequence of Proposition \ref{repr}, since the function field of any one-dimensional variety over $k$ has transcendance degree 1 over $k$.

\begin{prop}\label{prop taut}
Let $X$ be a one-dimensional component of irreducible type of~$X(M)$. Then there is an absolutely irreducible representation, denoted by $\rho_X: \pi_1(M) \to~\Sldeux(k(X))$, whose character is $\chi_X$. It is called the \textit{tautological representation} and it is defined up to $\GL_2(k(X))$-conjugation, 
\end{prop}

\subsection{Examples}
\label{subs:Examples}
In this subsection we compute the character varieties, together with a tautological representation, of the exteriors of the trefoil and the figure-eight knot.

\subsubsection{The trefoil knot}
Let $K$ be the trefoil knot in $\mS^3$, and $M$ be its complement $M= \mS^3\setminus K$. A presentation of its fundamental group is $\pi_1(M) =~\langle a,b \vert \ a^2 = b^3\rangle$. 

The element $z=a^2=b^3$ generates the center of $\pi_1(M)$. Any irreducible representation $\rho \colon \pi_1(M) \to \SL_2(k)$ needs to map $z$ into the center of $\SL_2(k)$ which is $\lbrace \pm \Id \rbrace$ (because $\rho(z)$ must commute with any $\rho(\gamma)$ for $\gamma$ in $\pi_1(M)$, and if $\rho(z) \neq \pm \Id$ it will contradict the irreducibility of $\rho$). 

If $\rho(z)=\Id$, then $\rho(a)=-\Id$ and necessarily $\rho$ is abelian, thus we can assume that $\rho(z)=-\Id$. Up to conjugacy, we can fix $\rho(b) =\bsm -j & 0\\ 0 & -j^2 \esm$ where $j$ is a primitive sixth root of unity. One can still conjugate $\rho$ by diagonal matrices without modifying $\rho(b)$, thus one can fix the right-upper entry of $\rho(a)$ to be equal to $1$ (if it was $0$, again, $\rho$ would be reducible).

Since $\rho(a)^2=-\Id$,  the Cayley-Hamilton theorem implies that $\Tr \rho(a) = 0$, hence $\rho(a) = \bsm t & 1\\-(t^2+1) & -t\esm$, for some $t \in k$.
As $(j-j^2)t = \Tr(ab^{-1})$, the function field of the component of irreducible type $X$ is $k(t)$;  and $X \simeq k$. The latter representation $\rho$ is the tautological representation.

\subsubsection{The figure-eight knot}
Let $M$ be the exterior of the figure-eight knot in $\mS^3$, $\pi_1(M)= \langle u,v \vert \ vw=wu \rangle$ with $w=[u,v^{-1}]$. Note that the meridians $u$ and $v$ are conjugated, hence they define the same trace functions. 

Denote by $x= Y_u = Y_v$, and by $y = Y_{uv}$. Then $B[\pi_1(M)] = k[x,y]/(P)$ where $P(x,y) = (x^2-y-2)(2x^2+y^2-x^2y-y-1)$ is obtained by expanding the relation $\Tr vwu^{-1} = \Tr w$ with the help of the trace relation (\ref{eq:TraceRelation}). The first factor of $P$ is the equation of the component of reducible type of $X(M)$ (compare with $\Delta_{u,v}$), and we denote by $X$ the curve of irreducible typer defined by the second factor of $P$. 
The tautological representation $\rho\colon \pi_1(M) \to k(X)[\al]/(\al+\al^{-1}-x)$ can be defined by $$\rho(u) = \bm \al & 1\\0 & \al^{-1} \ema , \rho(v)=\bm \al&0\\y-\al^2-\al^{-2} & \al^{-1} \ema$$
Although Proposition \ref{prop taut} ensures that the tautological representation can be defined directly with coefficients in $k(X)$, we do not know a simple expression of the latter with entries in this field.

\section{Culler-Shalen theory, group acting on trees and incompressible surfaces}\label{tree}
In this section we summarize a part of the Culler--Shalen theory, references include \cite{Sha02, Til03}. In their seminal articles \cite{CS83, CS84}, Marc Culler and Peter Shalen managed to use both tree-theoretical techniques introduced by Hyman Bass and Jean-Pierre Serre in \cite{Ser77} and character varieties to study the topology of 3-manifolds. In Subsections \ref{subs:Tree}, \ref{ends}, \ref{subs:stabilizers} we describe the Bass-Serre tree together with its natural $\Sldeux$ action, in Subsection \ref{valu} we recall some basic theory of algebraic curves and valuations. Finally in Subsection \ref{subs:Splitting} we prove some technical lemmas using Culler-Shalen theory.
\medbreak

\subsection{The tree}
\label{subs:Tree}
Let $\mK$ be an extension of $k$, we define a \textit{discrete $k$-valuation} as a surjective map $v: \mK \to \mZ \cup \lbrace \infty \rbrace$ such that 
\begin{itemize}
\item $v(0)= +\infty$
\item $v(x+y) \geq \operatorname{min}(v(x),v(y))$
\item $v(xy)=v(x)+v(y)$
\item $\forall z \in k, v(z)=0$ and $k$ is maximal for this property.
\end{itemize}
We will denote by $\cO_v=\lbrace x \in \mK\colon  v(x)\geq 0 \rbrace$ \textit{the valuation ring}, and we pick $t \in \cO_v$ an element of valuation 1, that we call a \textit{uniformizing parameter}. The group of invertible elements $\cO_v^*$ is the group of elements whose valuation is zero, $(t)$ is the unique maximal ideal of $\cO_v$ and $\cO_v/(t) \simeq k$ is the \textit{residual field}. Remark that every ideal is of the form $(t^n)$, for some $n\in \mathbb{N}$.

The main exemple to have in mind here is the valuation ring $\mC[[t]]$ of formal series in $t$, with the valuation $v:\mC((t))^*\to~\mZ$, $P\mapsto ~\mathrm{ord}_t(P)$ given by the vanishing order at $t=0$. Here the uniformizing element is $t$ and the residual field is $\mC$.

A \textit{lattice} $L$ in a two dimensional $\mK$-vector space $V$ is a free $\cO_v$-module of rank two that spans $V$ as a vector space.
The group $\mK^*$ acts on the set of lattices in $V$ by homothety . We denote by $T$ the set of equivalence classes of lattices where $L \sim L'$ if there exists $x \in \mK^*$ such that $L'=xL$. 

Now we define an integer-valued distance on $T$. We fix a lattice $L$ together with a basis of $L$, that turns $L$ into the standar lattice $\cO_v^2$. For any class $[L'] \in T$ one can express a basis of an element $L'$ of $[L']$ by a matrix $\bsm a&b\\c&d\esm$ with $ a,b,c,d \in \mK$. In fact, up to homothety we can pick $ a,b,c,d \in \cO_v$, that is $L' \subset L$.  Assume that $a$ has minimal valuation among $\lbrace a,b,c,d \rbrace$, then the matrix $\bsm a&b\\c&d \esm$ can be transformed into $\bsm a&0\\0&d-\frac{bc}{a}\esm$ by $\Sldeux(\cO_v)$ right and left multiplication that preserves the standard lattice $\cO_v^2$. Hence we have $L' \simeq~a\cO_v \oplus (d-~\frac{bc}{a}) \cO_v \simeq t^n\cO_v \oplus t^m \cO_v$ for some $n,m \in \mathbb{N}$. We define $d([L],[L']) =~\vert n-m \vert$. One can check that it defines a distance that does depend only on the equivalence classes of $[L]$ and of $[L']$.

This distance turns $T$ into a graph whose vertices are classes of lattices $[L]$ such that vertices at distance $1$ are linked by an edge. This graph is connected since any two vertices admit representatives $L,L'$ that can be described as $L = \cO_v \oplus \cO_v$ and $L'=\cO_v \oplus t^{d([L],[L'])}\cO_v$, hence a path joining $L$ to $L'$ can be constructed as $L_k = \cO_v \oplus t^k\cO_v$ with $k=0, \ldots, d([L],[L'])$. It can be shown that $T$ is indeed a tree, see \cite{Ser77}.

\subsection{Link of a vertex}\label{ends}
Given a vertex $[L] \in T$, one can describe the set of vertices at distance 1 of $[L]$ as follows: for each such $[L']$ there is a basis of $V$ such that $\cO_v^2$ is a representative of $[L]$ and that there is an unique representative $L'$ of $[L']$ isomorphic to $\cO_v \oplus t\cO_v$ in this basis. Since $tL \subset L'\subset L$, it defines a map 
$\lbrace [L']\colon  d([L],[L'])=1 \rbrace \to k\mP^1$ that sends $L'$ to the line $L'/tL$ in $L/tL \simeq k^2$, which turns out to be a bijection. 


\subsection{The $\Sldeux$ action: stabilizers of vertices, fixed points and translation length}
\label{subs:stabilizers}
There is a natural isometric and transitive action of $\operatorname{GL}(V)$ on $T$, induced by the action of $\operatorname{GL}(V)$ on $V$. We say that an element $g \in \GL(V)$ \textit{stabilizes} a vertex $[L] \in T$ if for any representative $L$ we have $g\cdot L=xL$ for some $x \in \mK^*$.

\begin{defi}
The action of a subgroup $G \subset \operatorname{GL}(V)$ on $T$ will be said \textit{trivial} if a vertex of $T$ is stabilized by the whole group $G$.
\end{defi}
\begin{lem}\label{det tree}
For any $g \in \Gldeux(\mK), [L] \in T$, fix a basis $\lbrace e,f \rbrace$ of $L \in [L]$, and $n,m \in \mZ$ such that $\lbrace t^n e, t^m f\rbrace$ is a basis of $g \cdot L$. Then $v(\det(g))=n+m$.
\end{lem}

\begin{proof}
In this basis, $L \simeq \cO_v^2$, and $g$ can be written as the matrix $A \bsm t^n & 0 \\ 0 & t^m \esm B$ with $ A,B \in \Gldeux(\cO_v)$. The result follows.
\end{proof}

Now we restrict to the $\operatorname{SL}(V)$ action. \begin{lem}
An element $g \in \operatorname{SL}(V)$ stabilizes a vertex $[L]$ iff for any representative $g\cdot L=L$.
\end{lem}
\begin{proof}
Assume that $g\cdot L= xL$, then by Lemma \ref{det tree}, $v(\det(g))=2v(x)=0$ hence $x \in \cO^*$ and $xL=L$.
\end{proof}
Since $\Sldeux(\cO_v)$ is the stabilizer of the standard lattice $\cO_v^2$, we deduce the following proposition.
\begin{prop}
\label{prop:Stabilizer}
The stabilizer in $\Sldeux(\mK)$ of any vertex of the tree $T$ is a $\operatorname{GL}_2$-conjugate of $\Sldeux(\cO_v)$.
\end{prop}

\begin{remk}
Since for any $g \in \operatorname{SL}(V)$, $v(\det(g))=0$, we know that the distance $d([L], g\cdot [L])=\vert n-m\vert$ is even, in particular $\operatorname{SL}(V)$ acts \textit{without inversion} on $T$, that is it can not fix an edge and exchange its end points.
\end{remk}

\begin{defi}
Given $g \in \operatorname{SL}(V)$, we define the \textit{translation length} of $g$ to be equal to $l(g)=\min\limits_{[L]\in T} d([L],g \cdot [L])$. 
\end{defi}

There are two ways for $g \in \operatorname{SL}(V)$ to act on the tree: 
\benu
\item \textbf{Elliptic elements}

If $g$ has some fixed points, or alternatively $l(g)=0$, then it will be called \textit{elliptic}. In this case, one can choose a basis of $V$ such that $g$ fixes the standard lattice $\cO_v^2$. Hence $g$ is (conjugated to) an element of $\Sldeux(\cO_v)$. The set of fixed points $T_g$ is a subtree of $T$.

\item \textbf{Hyperbolic elements}

If $l(g) > 0$, then $g$ is called \textit{hyperbolic}; $A_g=\lbrace s \in T \ \vert \ d(s,g \cdot s)= l(g) \rbrace$ is an infinite, globally fixed, axis on which $g$ acts by translation, and any basis of $V$ such that the standard lattice $\cO_v^2$ represents a point in  $A_g$ provides the matricial expression $g= \bsm t^{l(g)/2} & 0 \\ 0 & t^{-l(g)/2} \esm$. 

\eenu
A proof of the following lemma can be found in \cite[Corollaire 3, p.90]{Ser77}, see also \cite[Lemma 1.3.7]{BenThesis}
\begin{lem}\label{fix}
Let $G$ be a subgroup of $\operatorname{SL}(V)$ acting on the Bass-Serre tree $T$. If every element $g \in G$ fixes a vertex of $T$, then the whole group has a fixed vertex, that is the action is trivial.
\end{lem}

\subsection{Curves and valuations} \label{valu}
%
Examples of field extensions of $k$ together with $k$-valuations are given by algebraic varieties defined over $k$. In particular, pick $X \subset~X(M)$ an irreducible component of the character variety, its function ring $k[X]$ is a domain, and we denote by $k(X) = \operatorname{Frac}(k[X])$ its quotient field, called the \textit{function field} of $X$. It is a general fact that this field is a $k$-valuated field, with valuations corresponding to divisors $W \subset X$ (codimension one subvarieties). We will be interested in the case where $X$ is one-dimensional, and we refer to \cite{Ful} for details on what follows: there exists an unique curve $\h{X}$, which is smooth and compact, called the \textit{smooth projective model} of $X$, with a birational map $\nu\colon \h{X} \dashrightarrow X$ that is an isomorphism between open subsets and induces a canonical field isomorphism $\nu^*\colon  k(X) \xrightarrow{\sim} k(\h{X})$. There is a homeomorphism
\begin{align*}
\h{X} &\to \lbrace k\text{-valuations on } k(X) \rbrace \\
x &\mapsto v_x\colon  f \mapsto \operatorname{ord}_x(f)
\end{align*}
where the set of valuations is endowed with the cofinite topology.
\begin{remk}
When the context will be clear, a curve $X$ being given, we will often denote by $v$ a point in the smooth projective model $\h{X}$.
\end{remk}

\begin{defi}
Let $v \in \h{X}$ be a point in the projective model of $X$. We will say that $v$ is an \textit{ideal point} if $\nu$ is not defined at $v$, equivalently the function ring $k[X]$ is not a subring of $\cO_v$. Otherwise we will call $v$ a \textit{finite point}.
\end{defi}
\begin{ex}
Let $X$ be the plane curve $\lbrace x^2-y^3=0 \rbrace$ in $\mC^2$. It is a singular affine curve, with function ring $\mC[X] = \mC[U,V]/(U^2-V^3)$. The map $\mC[X] \to~\mC[T]$ that maps $U$ to $T^3$ and $V$ to $T^2$ induces a field isomorphism $\operatorname{Frac}(\mC[X]) \simeq ~\mC(T)$. Moreover, it defines a birational map $\nu\colon \mC\mP^1 \to X \subset \mC^2$ by $t \mapsto (t^3,t^2)$. Hence the smooth projective model of $X$ is isomorphic to $\mC\mP^1$ (remark that the singular point $(0,0)$ is "smoothed" through $\nu$), and the ideal point is $\infty$. Seen as a map $\mC\mP^1 \to~\mC\mP^2$, $\nu$~sends $[1 \colon0]$ to $[1\colon0\colon0]$, the curve $X \cup \lbrace \infty \rbrace \subset \mC\mP^2$ is a compactification of $X$.
\end{ex}
\subsection{Group acting on a tree and splitting}
\label{subs:Splitting}
In this subsection we summarize how the Culler--Shalen theory associates an incompressible surface $\Sigma \subset M$ to an ideal point in the character variety of $M$.
\medbreak

Recall that in this paper $M$ is a $3$-manifold whose boundary is a torus. We assume that $X$ is a one-dimensional irreducible component of irreducible type of the character variety $X(M)$. Let $\rho\colon \pi_1(M) \to~\Sldeux(k(X))$ be a choice of a tautological representation. Let $v\in \h{X}$ be a point in the smooth projective model of $X$, the pair $(k(X),v)$ is a $k$-valuated field, and we denote by $T_v$ the Bass-Serre tree described above. The group $\pi_1(M)$ acts simplicially on $T_v$ as a subgroup of $\Sldeux(k(X))$ through the tautological representation $\rho$. Note that although the tautological representation is defined up to conjugation, the action of $\pi_1(M)$ on the Bass-Serre tree is well-defined.

\begin{lem}\label{split}
The action of $\pi_1(M)$ on $T_v$ is trivial if and only if $v \in \h{X}$ is a finite point.
\end{lem}
\begin{proof}
By definition, for $v \in \h{X}$ a finite point, the ring $k[X]$ is included in $\cO_v$, which means that $v(Y_\g) \ge 0$ for any $\g \in \pi_1(M)$. Equivalently, $\Tr(\rho(\g)) \in \cO_v$ for any $\g \in \pi_1(M)$, and we want to prove that it is equivalent to $\rho(\g)$ to be conjugated to an element of $\Sldeux(\cO_v)$. It is clear if $\rho(\g) = \pm \Id$, if not there exists a vector $e \in k(X)^2$ such that $\lbrace e, \rho(\g) e \rbrace$ is a basis of the two dimensional vector space $k(X)^2$, and in this basis $\rho(\g)$ acts as the matrix $\bsm 0 & 1\\-1 & \Tr(\rho(\g)) \esm \in \Sldeux(\cO_v)$. Hence we have proved that $v$ is a finite point if and only if for all $\g \in \pi_1(M)$ one has $\rho(\g) \in \Sldeux(\cO_v)$. The proposition follows now from Lemma \ref{fix}.
\end{proof}

It motivates the following definition. 
\begin{defi}
A representation $\rho \colon \pi_1(M) \to \SL_2(k(X))$ whose entries lie in the valuation ring $\cO_v$ will be said \textit{convergent at $v$}. Given a convergent representation $\rho\colon \pi_1(M) \to \SL_2(\cO_v)$, we will denote by $\brho$ the composition $$\brho \colon \pi_1(M) \to \Sldeux(\cO_v) \xrightarrow{\text{mod }t} \Sldeux(k)$$ and we will call it the \textit{residual} representation. 
\end{defi}

Combining Proposition \ref{prop:Stabilizer} and Lemma \ref{split} we obtain immediately the proposition:
\begin{prop}
\label{prop:Convergent}
The tautological representation can be chosen, up to conjugation, to be convergent at a point $v$ in the smooth projective model $\hat{X}$ of $X$ if and only if $v$ is a finite point. In this case it can be chosen as a representation $\rho\colon \pi_1(M) \to \Sldeux(\cO_v)$. 
\end{prop}
Notice that if $v$ corresponds to the character $\chi \in X$, then the residual representation $\brho$ is a lift of $\chi$.

\subsubsection{Ideal points and incompressible surfaces}
Let $v$ be an ideal point in $\h{X}$. We know from Proposition \ref{split} that no representative $\rho$ of the tautological representation can be chosen to be convergent. Now we describe quickly how to construct, from the action of $\pi_1(M)$ on $T_v$ a surface $\Si \subset M$, said \textit{dual} to the action. The reader will find many details about this delicate construction in \cite{Sha02, Til03}.
\medbreak

The main point is to construct a $\pi_1(M)$-equivariant map $f\colon \ti{M} \to T_v$. Pick any triangulation $K$ of $M$, and lift it to a $\pi_1(M)$-invariant triangulation $\ti{K}$ of $\ti{M}$. Then pick a set of orbit representatives $S^{(0)}$ for the action of $\pi_1(M)$ on the set of $0$-simplices of $\ti{K}$, and any map $f_0\colon S^{(0)} \to T_v$ from this set to the set of vertices of $T_v$. It induces an equivariant map from the $0$-squeleton of $\ti{K}$ to $T_v$, that we still denote by $f_0\colon \ti{K}^{(0)} \to T_v$. Now it is possible to extend simplicially this map to the $1$-squeleton, as follows: pick a set of orbit representatives $S^{(1)}$ for the action of $\pi_1(M)$ on the set of $1$-simplices of $\ti{K}$. Any edge $\sigma \in S^{(1)}$ has endpoints mapped to some given vertices through the map $f_0$, and we extend in the obvious way $f_0$ to $\sigma$. Now there is a unique $\pi_1(M)$-equivariant extension $f_1\colon \ti{K}^{(1)} \to T_v$ of $f_0$, it is continuous, and can be made simplicial, up to subdivide the triangulation $\ti{K}$. Repeat this process up to obtain the desired simplicial, equivariant map $f\colon \ti{M} \to T_v$.

Now consider the set of midpoints $E$ of the edges of $T_v$, the set $f^{-1}(E)$ is a surface $\ti{S} \subset \ti{M}$. This surface is non-empty because the action of $\pi_1(M)$ on the tree $T_v$ is non-trivial, and orientable because the map $f$ is transverse to $E$. Moreover it is stable under the action of $\pi_1(M)$ on $\ti{M}$, and hence its image through the covering map $\ti{M} \to M$ is a surface $S \subset M$, non-empty and orientable, dual to the action. It is worth to notice that it has no reason to be connected in general.

\begin{defi}
A surface $\Si$ in a 3-manifold $M$ is said \textit{incompressible} if
\benu
\item
$\Si$ is oriented
\item
For each component $\Si_i$ of $\Si$, the homomorphism $\pi_1(\Si_i) \to \pi_1(M)$ induced by inclusion is injective.
\item
No component of $\Si$ is a sphere or is boundary parallel.
\eenu
\end{defi}
\begin{remk}
A \textit{compression disk} $D \subset M$ is an embedded disk in $M$ such that $\partial D$ lies in $S$ and is not homotopically trivial in $S$. The second condition above is equivalent to saying that there is no compression disk in $M$.
\end{remk}

If $S$ is a surface dual to a $\pi_1$ action on a tree $T$, there is a way to modify the equivariant map $f$ in order to avoid compression disks, spherical and boundary parallel components, and hence to obtain a new surface $\Si$ which  is incompressible. We refer the reader to the references given above, where a proof of this fact can be found.
\begin{remk}
\label{remk:SigmaFaithfull}
The surface group $\pi_1(\Sigma_i)$ of each connected component of $\Sigma$ acts faithfully on the tree $T_v$ by construction, and identifies with the stabilizer of an edge. In particular the tautological representation restricts to a faithfull representation $\rho_\Sigma \colon \pi_1(\Sigma) \to \SL_2(\mC(X))$ (note that if $X$ contains the character of a faithfull representation, then the whole representation $\rho \colon \pi_1(M) \to \SL_2(\mC(X))$ is faithfull). In particular, it will be used in the sequel that if $\pi_1(\Sigma)$ is not abelian, then there exists $\gamma$ in $\pi_1(\Sigma)$ with $\Tr \rho(\gamma) \neq 2$.
\end{remk}
\subsubsection{The separating case.} \label{prime}
Let $\Si$ be an incompressible surface associated to an ideal point $v\in \h{X}$. In this section we suppose that $\Si$ is a union of $n$ parallel copies $\Si_i, i=1, \ldots, n$ and that each copy is separating $M$ into two handlebodies $M=M_1\cup_{\Si_i} M_2$. Consider $V(\Si)$ a neighborhood of $\Si$ homeomorphic to $\Si_1 \times [0,1]$, and we consider the splitting $M = M_1 \cup_{V(\Si)} M_2$. We fix a basepoint $p \in \Si_1$, and we will denote by $\pi_1(\Si)$ the fundamental group of $\Si_1$ based in $p$. We identify $\pi_1(V(\Si))$ to $\pi_1(\Si)$, and the Seifert-Van Kampen Theorem provides the amalgamated product $\pi_1(M)=~\pi_1(M_1)\ast_{\pi_1(\Sigma)}\pi_1(M_2)$. A sketchy picture is drawn in Figure \ref{split surface}.
\begin{figure}[h]
\begin{center}
\def\svgwidth{0.6\columnwidth}
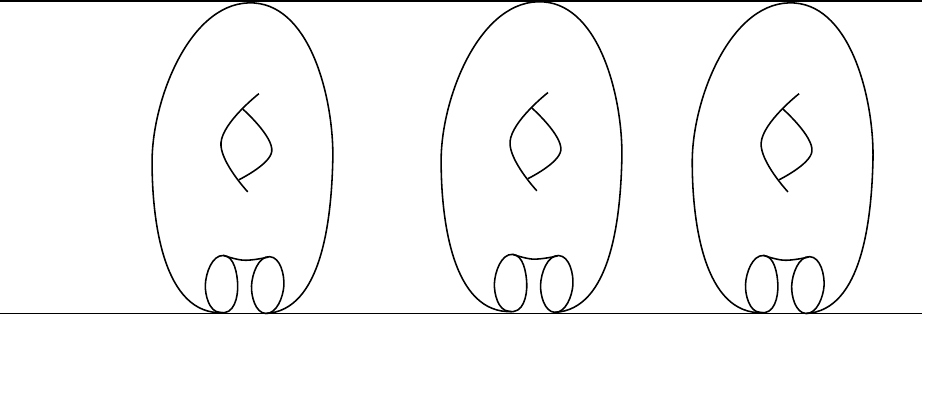
\caption{\label{split surface} The splitting $M = M_1 \cup_{V(\Si)} M_2$.}
\end{center}
\end{figure}
\begin{lem} \label{converge}
There exists a representative of the tautological representation, denoted by $\rho\colon \pi_1(M) \to\Sldeux(k(X))$, satisfying the following:
\benu
\item
The restriction $\rho_1 \colon \pi_1(M_1) \to \SL_2(k(X))$ is convergent, as well is the composition $$\rho_{1,\Sigma} \colon \pi_1(\Sigma) \to \pi_1(M_1) \xrightarrow{\rho_1} \SL_2(\cO_v).$$

\item
If $\rho_2\colon \pi_1(M_2) \to \SL_2(k(X))$ is the restriction of $\rho$ to $\pi_1(M_2)$, then there is a convergent representation $\rho_2'\colon \pi_1(M_2)\to \Sldeux(\cO_v)$ such that $\rho_2 = U_n \rho_2' U_n^{-1}$, with $U_n = \bsm t^n & 0\\0 & 1 \esm$. It induces an other convergent representation of the fundamental group of $\Sigma$ that we denote by $$\rho_{2, \Sigma} \colon \pi_1(\Sigma) \to~\pi_1(M_2) \xrightarrow{\rho'_2} \SL_2(\cO_v).$$
\item
The residual representations $\brho_{1,\Sigma} \colon \pi_1(\Sigma) \to \pi_1(M_1) \to \SL_2(\cO_v) \to \SL_2(k)$ and $\brho_{2,\Sigma} \colon \pi_1(\Sigma) \to \pi_1(M_2)\to \SL_2(\cO_v) \to \SL_2(k)$ obtained by restriction to $\pi_1(\Sigma)$ of either $\rho_1$ or $\rho'_2$ are reducible.
\eenu
 \end{lem}
\begin{proof}
Let $s_1 \in T_v$ be a vertex in the Bass-Serre tree which is fixed by $\rho(\pi_1(M_1))$. Fix a basis such that it corresponds to the lattice $\cO_v^2$. Then there is a vertex $s_2 \in T$, fixed by $\rho(\pi_1(M_2))$, such that $d(s_1,s_2)=n$. Moreover, assume that in this basis $s_2$ has a representative of the form $t^n\cO_v \oplus \cO_v$. The first observation is that $\rho_1(\pi_1(M_1)) \subset \Sldeux(\cO_v)$ because it stabilizes $\cO_v^2$, hence $\rho_1$ is convergent.

Let $\rho'_2 = U_n^{-1} \rho_2 U_n$, then $\rho'_2 \cdot s_1  = U_n^{-1} \rho_2 \cdot s_2 = U_n^{-1} \cdot s_2 = s_1$ and we have proved that the representation $\rho'_2$ also converges.

Since $\rho_1(\pi_1(\Sigma))$ fixes the first edge of the segment $[s_1s_2]$, in this basis it fixes the lattices $\cO_v^2$ and $t\cO_v \oplus \cO_v$, hence for all $\g \in \pi_1(\Sigma)$, $\rho_\Sigma(\g) = \bsm a(\g)& b(\g)\\c(\g) & d(\g) \esm$, with $c(\g) \in (t)$, hence $\brho_{1,\Sigma}$ is reducible, and the reducibility of $\brho_{2, \Sigma}$ follows in the same way.

\end{proof}
\subsubsection{The non-separating case} \label{non-split section}
Let $\Sigma$ be an incompressible surface associated to an ideal point $v \in \h{X}$ which is, again, union of $n$ parallel copies $\Sigma=\Sigma_1 \cup \ldots \cup \Sigma_n$. We assume now that $M \setminus \Sigma_i$ is connected for any $i$.
Let $V(\Sigma_i)$ be a neighborhood of $\Sigma_i$ in $M$, and $E(\Sigma_i) = \overline{M \setminus V(\Sigma_i)}$. It is proved in \cite[Proposition 2]{Oza01} that $E(\Sigma_i)$ is a handlebody if and only if $\pi_1(E(\Sigma_i))$ is free. In this case we say that the surface $\Sigma_i$ is \textit{free}. It is the case, for instance, if $M$ is a small manifold (a manifold that does not contain any closed incompressible surfaces). 

If $M$ is the exterior of a knot $K$ in an integer homology sphere, then any $[\Sigma_i]$ is a generator of $H_2(M, \partial M) \simeq \mZ$ and any connected component $\gamma \subset \partial \Sigma$ is null-homologous. In particular if $\partial \Sigma_i$ is connected, then $\Sigma_i$ is a Seifert surface for $K$. A necessary and sufficient condition for a knot exterior $M$ to contain non-free Seifert surfaces is given in \cite{Oza99}. 

In the sequel we assume that the surface $\Sigma_i$ is free, say of genus $g$, and we denote by $H=E(\Sigma_i)$ the genus $2g$ handlebody complement of $\Sigma_i$. We assume that $\partial V(\Sigma) =~\Sigma_1 \cup \Sigma_n$. We have $M = V(\Sigma) \cup_{\Sigma_1 \cup \Sigma_n} H$, it provides the HNN decomposition $\pi_1(M) = \pi_1(H) \ast_\al$. Here we fix the basepoint $p$ in $\Sigma_1$, and $\al\colon \pi_1(\Sigma_1) \to \pi_1(\Sigma_n)$ an isomorphism between those subgroups of $\pi_1(M)$. This means that we have the presentation $\pi_1(M) = \langle \pi_1(H), v \ \vert \ v \g v^{-1} = \al(\g), \forall \g \in \pi_1(\Sigma_1) \rangle$.

\begin{lem}
\label{lem:NonSplit}
There exists a representative of the tautological representation, denoted by $\rho\colon \pi_1(M) \to \Sldeux(k(X))$ satisfying the following:
\benu
\item
The restrictions $\rho_H:\pi_1(H)\to \Sldeux(k(X))$ and $\rho_1\colon \pi_1(\Sigma_1) \to \Sldeux(k(X))$ are convergent. 
\item
The representation $\rho$ maps $v$ to $\rho(v) = V_n= \bsm t^\frac{n}{2} & 0\\0 & t^\frac{-n}{2} \esm$ in $\SL_2(k(X))$, in particular $n$ is even.
\item
The restriction of the tautological representation $\rho_n\colon \pi_1(\Sigma_n) \to \Sldeux(k(X))$ is equal to $\check{U}_n^{-1} \rho_1 \check{U}_n$, with $\check{U}_n = \bsm 1 & 0\\0 & t^n\esm$.
\item
The residual restricted representation $\brho_1$ has the following form: for all $\gamma$ in $\pi_1(\Sigma)$, $\brho_1(\gamma)=\bsm \lambda(\gamma) & 0 \\ * & \lambda^{-1}(\gamma)\esm$ for some $\lambda \colon \pi_1(\Sigma)\to k^*$. In particular $\brho_1$ is reducible.
\eenu
\end{lem}

\begin{proof}
We fix a vertex $s_1$ in the Bass-Serre tree $T$, that corresponds to the lattice $\cO_v^2$ and is fixed by $\pi_1(H)$. Hence $\rho_H$ is convergent. We denote by $e_1$ the edge in the tree $T$ incident to $s_1$ that is fixed by $\pi_1(\Sigma_1)$, and the parallel copies of $\Sigma_1$ stabilize a series of edges $e_i$ that form a segment in $T$, which has $V_n \cdot s_1$ as an end point, as depicted on the fundamental domain of the tree below.
\begin{center}
\begin{tikzpicture}
\node[circle,fill,label=$s_1$,inner sep=1pt] (s0) at (0,0){};
\node[circle,fill,inner sep=1pt] (s1) at (1.5,0){};
\node[circle,fill,inner sep=1pt] (s2) at (3,0){};
\node[] (s) at (4.5,0){...};
\node[circle,fill,inner sep=1pt] (s3) at (6,0){};
\node[circle,fill,label=$V_n \cdot s_1$,inner sep=1pt] (sn) at (7.5,0){};
\draw (s0)--(s1) node[below left=0.4cm]{$e_1$};
\draw (s1)--(s2) node[below left=0.4cm]{$e_2$};
\draw (s3)--(sn) node[below left=0.4cm]{$e_n$};
\draw (s2)--(s)--(s3);
\end{tikzpicture}
\end{center}
In particular, the representation $\brho_1$ is reducible with the claimed form since it stabilizes both the lattices $\cO_v^2$ and $t\cO_v \oplus \cO_v$. Now the element $v \in \pi_1(M)$ acts on the tree $T$ as a hyperbolic element of translation length equal to $n$: it sends the vertex $s_1$ to $V_n \cdot s_1$ and in particular it has the claimed form. 

By construction, since $\pi_1(\Sigma_n) = v\pi_1(\Sigma_1)v^{-1}$ we have $\rho_n = V_n \rho_1V_n^{-1}$. Note that the action by conjugation of the matrix $V_n$ is the same than the action of the matrix $\check{U}_n^{-1}$, and item 3) follows. It completes the proof.

\end{proof}
\section{The Reidemeister torsion} \label{Reidem}
In this section we define the Reidemeister torsion of a complex (\ref{sub:DefiTorsion}) and use this definition to construct the torsion function as a rational function on the geometric component of the character variety (\ref{sub:Torsion}).
\subsection{Reidemeister torsion}
\label{sub:DefiTorsion}
We give various definitions used for the Redemeister torsion.
References are \cite{Mil66}, \cite[Appendix A]{GKZ94}, \cite[Chapitre 0]{Porti97}. We stress out the fact that we use a convention (namely, how we take the alternating sum in the definition of the determinant of a complex) that corresponds to \cite{GKZ94}, but not to \cite{Mil66}.
\medbreak

\subsubsection{Definition of the torsion}
\label{subsub:Definition}

Given a finite complex $C^*$ of $k$-vector spaces  
$$C^0 \xrightarrow{d_0} C^1 \xrightarrow{d_1} ... \xrightarrow{d_{n-1}} C^n$$
with $\lbrace c^i \rbrace_{i=0...n}$ and $\lbrace h^i \rbrace_{i=0...n}$ families of bases of the vector spaces $C^i$'s and $H^i$'s, one can define the \textit{torsion} of the based complex $\tor(C^*,\lbrace c^i \rbrace, \lbrace h^i\rbrace)$ as an alternating product of determinants. More precisely, consider the exact sequences
$$0 \to Z^i \to C^i \xrightarrow{d_i} B^{i+1} \to 0$$
$$0 \to B^i \to Z^i \to H^i \to 0$$
that define the vector spaces $B^i$, $Z^i$ and $H^i$.
Pick a system of bases $\lbrace b^i\rbrace$ of the $B^i$'s, first one obtains a basis $b^i \sqcup \bar{h}^i$ of $Z^i$ for any $i$, given by any choice of a section $H^i \to Z^i$. Then any section $B^{i+1} \to C^i$ provides a basis of $C^i$ denoted by $ b^i\sqcup \bar{h}^i \sqcup \bar{b}^{i+1}$. Now compare this new basis with the original basis $c^i$, and take the determinant of the change of basis matrix, denoted by $[b^i\sqcup \bar{h^i} \sqcup \bar{b}^{i+1}: c^i]$. One can show that the alternating product of those determinants does not depend on the lifts, neither on the choice of basis $\lbrace b^i \rbrace$. We define 
 $$\tor(C^*,\lbrace c^i \rbrace, \lbrace h^i \rbrace)= \prod_i [b^i\sqcup \bar{h}^i \sqcup \bar{b}^{i+1}: c^i]^{(-1)^i} \in \mK^*/\lbrace \pm 1 \rbrace$$
 
\begin{remk}
It can be seen as a generalization of the determinant: in particular if the complex is just an isomorphism $(C^0,c^0) \xrightarrow{d_0} (C^1,c^1)$, then the torsion $\tor(C^*,\lbrace c^* \rbrace, \emptyset)$ is nothing but the determinant of the invertible matrix of the map $d_0$ in the bases $c^0$ and $c^1$. Note that we defined the torsion up to sign indeterminacy. For the use we will make in this article it makes no difference because we want to study vanishing properties of the torsion, nevertheless we stress out that this sign indeterminacy can be solved in our setting, for instance in \cite{Dub06, DHY09}.
\end{remk}

\subsubsection{The Euler isomorphism}
Given $V$ an $n$-dimensional $\mK$-vector space, its determinant vector space $\det(V)=~\bigwedge^nV$ is defined as its $n$-th exterior power. It is a one-dimensional vector space: if $\lbrace v_1, \ldots, v_n\rbrace$ is a basis of $V$, there is an isomorphism $\det(V) \to \mK$ obtained by sending the vector $v_1\wedge \ldots \wedge v_n$ to $1$. In the following, for $L$ a one-dimensional vector space, the notation $L^{\otimes(-1)}$ will denote the dual vector space $\Hom(L,\mK) =L^*$. 
One has the following lemma:
\begin{lem}\cite[Appendix A, Lemma 5]{GKZ94}
\label{lem:Determinant}
Let $0 \to A \to B \to C \to 0$ be an exact sequence of vector spaces, then there is a natural isomorphism $$\det(A)\otimes \det(C) \simeq \det(B).$$
\end{lem}

\begin{defi}
Let $V^*=\bigoplus V^i$ be a finite dimensional graded $\mK$-vector spaces. The \textit{determinant} of $V^*$ is defined by 
$$\det(V^*) = \bigotimes\limits_i \det(V^i)^{\otimes(-1)^i}.$$
\end{defi}
Given a complex $C^*$, the cohomology of this complex is naturally graded, and we have the following proposition that follows easily from Lemma \ref{lem:Determinant}:
\begin{prop}\cite[Appendix A, Proposition 3]{GKZ94}
\label{prop:Euler}
There is a natural (Euler) isomorphism $\operatorname{Eu} \colon \det(C^*) \xrightarrow{\sim} \det(H^*(C^*))$.
\end{prop}

Now we fix a based complex $(C^*, c^*)$, where for any $i$, the basis is denoted by $c^i =~\lbrace c_1^i, \ldots, c_{n_i}^i \rbrace$. Then we denote by $\bigwedge c^i$ the basis element $c_1^i\wedge \ldots \wedge c_{n_i}^i$ of $\det(C^i)$, and by $\textbf{c} = \bigotimes_i(\bigwedge c^i)^{\otimes(-1)^i}$ the induced basis of the vector space $\det(C^*)$. 
\begin{defi}
\label{defi:Torsion}
The torsion of the based complex $(C^*,c^*)$ is 
$$\tor(C^*,\textbf{c}) = \operatorname{Eu}(\textbf{c}) \in \det(H^*(C^*))/\lbrace\pm1\rbrace.$$
\end{defi}
\begin{remk}
\label{remk:DefinitionsTorsion}
The notation is meaningful: in fact the torsion does not depend on the basis $c^*$ of the complex $C^*$, but only on the basis element $\textbf{c}$ of $\det(C^*)$. Moreover, it coincides with the definition provided in \ref{subsub:Definition} in the following sense:
if $h^*$ is a basis of the graded vector space $H^*(C^*)$, then it defines a basis element $\textbf{h}=~\bigotimes_i (\bigwedge h^i)^{\otimes(-1)^i}$, and one can compare $\operatorname{Eu}(\textbf{c})$ with $\textbf{h}$ in $\det(H^*(C^*))$. It provides an element of $\mK$, that we denote by $[\operatorname{Eu}(\textbf{c}):\textbf{h}]$, and we have
$$\tor(C^*,\lbrace c^i\rbrace,\lbrace h^i \rbrace) = [\operatorname{Eu}(\textbf{c}):\textbf{h}].$$
\end{remk}
\subsubsection{The Cayley formula}
\label{subsub:Cayley}
When $C^*$ is an exact complex, a first occurence of a description of the torsion can be found in the seminal work of Cayley in 1848 (see \cite[Appendix B]{GKZ94} where the original text is retranscribed).

Let $(C^*, c^*)$ be a based complex of $\mK$-vector space of the form
$$0 \to (C^0,c^0) \xrightarrow{d_0} \ldots \xrightarrow{d_{r-1}} (C^r,c^r) \to 0.$$ Assume that this complex is exact: it has trivial homology. In particular the one-dimensional vector space $\det(H^*(C^*))$ is canonically isomorphic to $\mK$. We abusively denote by $\tor(C^*, \textbf{c})$ the element of $\mK^*/\lbrace \pm 1\rbrace$ given by $[\tor(C^*, \textbf{c}):1]$, then the equality of Remark \ref{remk:DefinitionsTorsion} $$\tor(C^*, \lbrace c^* \rbrace, \emptyset) = \tor(C^*, \textbf{c})$$ between the two definitions of the torsion holds.

For each index $i$, the basis $c^i = \lbrace c^i_1, \ldots, c^i_{n_i} \rangle$ can be partitioned into two subsets $c^i_I$ and $c^i_J$ such that $\ker d_i = \langle c^i_I \rangle$. Hence we have $C^i =  \langle c^i_I \rangle \oplus  \langle c^i_J \rangle$ and the map $d_i$ restricts to a linear isomorphism $(d_i)_{I,J}\colon \langle c^i_J \rangle \to \operatorname{im}(d_i) = \langle c^{i+1}_I \rangle$ whose determinant we denote by $\Delta_i$. Of course, those determinants depend on the choices, but it can be shown that their alternated product do not, and we have the proposition:
\begin{prop}\cite[Appendix A, Theorem 14]{GKZ94}
\label{prop:Cayley}
With the preceding notations, the following equality holds:
$$\tor(C^*, \textbf{c}) = \prod\limits_{i=0}^{r-1} \Delta_i^{(-1)^{r-i-1}} \in \mK^*/\lbrace \pm1\rbrace.$$
\end{prop}

\subsection{The torsion function}
\label{sub:Torsion}
Let $X$ be a one-dimensional component of irreducible type of the character variety $X(M)$. In Section \ref{taut} we defined the tautological representation $\rho\colon \pi_1(M) \to \Sldeux(k(X))$, up to conjugation. In particular the torsion of the twisted complex $C^*(M, \rho)$ of $k(X)$-vector spaces is well-defined. It provides a non-zero, defined up to sign, element of the homological determinant vector space \[\tor(M, \rho) \in \det(H^*(M, \rho)) \setminus \lbrace0\rbrace.\] 
We prove the following propostion:
\begin{prop}
If the complex $C^*(M, \rho)$ is acyclic, then the Reidemeister function $\tor(M, \rho) \in k(X)^*$ defines a rational function on the curve $X$.
In particular, it is the case if $M$ is hyperbolic and $X$ contains the character of a lift of a holonomy representation.
\end{prop}
\begin{proof}
Since $C^*(M, \rho)$ is acyclic the determinant $\det(H^*(M, \rho))$ is canonically isomorphic to $k(X)^*$, as claimed at first.

Assume that $M$ is hyperbolic and $X$ contains the character of a lift of a holonomy representation, and let us prove that in this case the complex $C^*(M, \rho)$ is acyclic.

Since $M$ has the same homotopy type as a two-dimensional CW complex, it has no homology in rank greater that 2.

The space of invariants vectors $H^0(M, \rho) =\lbrace z \in k(X)^2 \vert \rho(\g) z=z, \forall \g \in \pi_1(M) \rbrace$ is non-trivial if and only if $\Tr(\rho(\g)) = 2$ for all $\g \in \pi_1(M)$. As the tautological representation is irreducible, there exists an element $\gamma$ in $\pi_1(M)$ such that $\Tr \rho(\gamma) \neq 2$, hence $H^0(M, \rho) = 0$. 

We know that the Euler characteristic $\chi(M)$ is zero, hence it is now enough to prove that $H^1(M, \rho) = 0$. The universal coefficients theorem provides isomorphisms $H^1(M, \rho)_v \otimes k(X) \simeq H^1(M, \rho)$ and $H^1(M, \brho) \simeq H^1(M, \rho)_v \otimes k$, hence it is enough to show that for some $\chi \in X$, one has $H^1(M, \brho) = 0$. It follows from Ragunathan's vanishing theorem (see for instance \cite[Theorem 0.2] {MP10}) that it is the case if $\chi$ is the character of a holonomy representation.
\end{proof}

\begin{remk}
As soon as there exists a character $\chi \in X$ such that $H^1(M, \brho)$ is trivial, the proposition applies and the torsion defines a well-defined function on the curve $X$. If follows from the semi-continuity of the dimension of $H^1(M, \brho)$ on $X$ that in this case, $H^1(M, \brho)$ is trivial for all but a finite numbers $\chi \in X$. It has been the way to define almost everywhere the torsion function on $X$, the novelty here is that it is defined \textit{a priori} as a rational function, even at characters $\chi$ with non-trivial first cohomology groups. Indeed, we will show in the next section that the torsion function vanishes exactly at those points.
\end{remk}
\section{The case of a finite character}
\label{sec:Finite}
In this section we prove Theorem \ref{acyclic finite} in Subsection \ref{sub:Thm1}, then we give a series of examples (Subsection \ref{section examples})
\medbreak

\subsection{Proof of Theorem \ref{acyclic finite}}
\label{sub:Thm1}
In this subsection we prove Theorem \ref{acyclic finite} of the introduction.
\medbreak

Let $X$ be a one-dimensional component of irreducible type of the character variety $X(M)$ and $\rho\colon \pi_1(M) \to \Sldeux(k(X))$ the tautological representation. Let $v$ be a finite point of $X$, then Proposition \ref{prop:Convergent} allows us to pick a convergent representative of the tautological representation $\rho \colon \pi_1(M) \to \SL_2(\cO_v)$.
In this section, and in the rest of this paper, we will be frequently interested in the twisted complex of $\cO_v$-modules $C^*(M, \rho)_v$ induced by this restriction of coefficients. In particular we will denote the corresponding cohomology $\cO_v$-modules by $H^*(M, \rho)_v$. 
\begin{lem}
\label{lem:TorsionModules}
The modules $H^*(M, \rho)_v$ are torsion modules.
\end{lem}
\begin{proof}
Since $C^*(M, \rho)$ is acyclic, the $k(X)$-vector spaces $H^i(M, \rho) \simeq H^i(M, \rho)_v \otimes_{\cO_v} k(X)$ are trivial. Since $\cO_v$ is a domain with fraction field $k(X)$, the result follows.
\end{proof}
Such a torsion $\cO_v$-module $N$ can be written as a finite direct sum $\oplus_i \cO_v/(t^{n_i})$, and we define the length of $N$ by $\len(N) = \sum_i n_i$.

Whenever such a convergent tautological representation $\rho \colon \SL_2(\cO_v)$ is given, it makes also sense to consider the residual complex of $k$-vector spaces $C^*(M, \brho)$, and we will also use those residual cohomology $k$-vector spaces $H^*(M, \brho)$.
\begin{theo}
\label{theo:Finite}
If $v$ is a finite point of $X$, the vanishing order of the torsion function at $v$ is given by:
$$v(\tor(M, \rho) = \len (H^2(M, \rho)_v)$$
In particular it vanishes if and only if $H^1(M, \brho)$ is non-trivial, where $\brho$ is the residual representation $\brho\colon \pi_1(M) \xrightarrow{\rho} \Sldeux(\cO_v) \xrightarrow{\text{ mod } t} \Sldeux(k)$.
\end{theo}
We deduce immediately the following theorem:
\begin{theo}
The torsion function $\tor(M, \rho)$ is a regular function on $X$, that is an element of the function ring $k[X]$.
\end{theo}

For a complex $C^*$  of $\cO_v$-modules such that $C^* \otimes k(X)$ is an exact complex, we will say that $C^*$ is rationally exact.
The main tool of the proof of Theorem \ref{theo:Finite} is the following theorem which is proved in \cite[Theorem 30]{GKZ94} or \cite{Tur01}.

\begin{theo}
\label{theo:Russes}
Let $\chi \in X$ a character, and $v$ a valuation on $k(X)$ associated to $\chi$. If $C^*$ is a rationally exact based complex of $\cO_v$-modules with basis $\lbrace c^i \rbrace$, then 
$$v(\tor(C^* \otimes k(X), \lbrace c^i \rbrace ))= \sum_k (-1)^{k} \len(H^k(C^*))$$
\end{theo}
We will need the following lemma, see \cite[Lemma 3.9]{Porti97} or \cite[Lemma 2.16]{Ben16} for a proof.
\begin{lem} \label{lem central}
Assume that $M$ is a knot in a rational homology sphere. Let $\chi \in X$ be a reducible character in a component of irreducible type of the character variety $X(M)$. Then the character $\chi$ is non-central.
\end{lem}

\begin{proof}[Proof of Theorem \ref{theo:Finite}]
Since the complex $C^*(M, \rho)$ is acyclic, we can apply Theorem \ref{theo:Russes}. Recall from Lemma \ref{lem:TorsionModules} that the $H^i(M, \rho)_v$ are torsion $\cO_v$-modules. As a submodule of a free module, $H^0(M, \rho)_v$ is trivial. Then Lemma \ref{lem central} implies that no character $\chi \in X$ is central, in particular $H^0(M, \brho)$ is trivial. But the universal coefficients theorem provides the isomorphisms $H^0(M, \brho) \simeq H_0(M, \brho)^*$, and $H_0(M, \brho) \simeq H_0(M, \rho)_v \otimes k$, thus we have proved that $H_0(M, \rho)_v$ is trivial. Again by the universal coefficients theorem we have $ \Ext(H^1(M, \rho)_v, \cO_v) \simeq H_0(M, \rho)_v = \lbrace 0 \rbrace$, and we conclude that $H^1(M, \rho)_v \simeq \Ext(H^1(M, \rho)_v, \cO_v) = \lbrace 0 \rbrace$ because it is a torsion module. In conclusion we have proved the first part of the theorem
$$v(\tor(M, \rho)) = \len(H^2(M, \rho)_v)$$
Now $H^2(M, \rho)_v$ being trivial is equivalent to  $H^2(M, \brho)$ being trivial which is the same that $H^1(M, \brho)$ being trivial, and the theorem is proved.
\end{proof}

\subsection{Some computations and examples} \label{section examples}
We compute the torsion function on a series of examples of twist knots, and determine its zeros on the character variety.
\begin{figure}[h]
\begin{center}
\def\svgwidth{0.2\columnwidth}
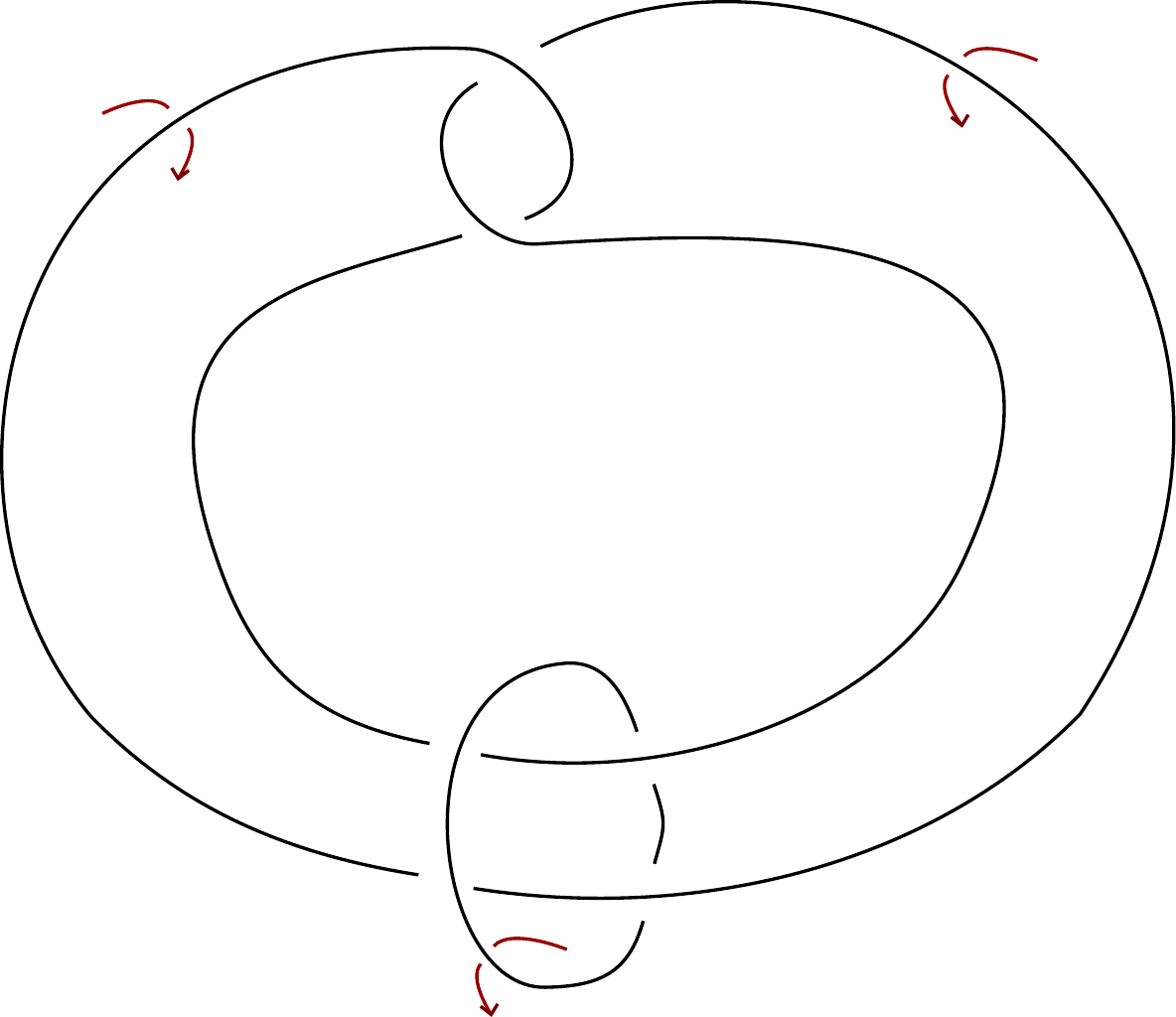
\caption{\label{whitehead}A diagram of the Whitehead link. $a,b$ and $\lambda$ are depicted generators of the fundamental group, and $\mu$ is a counter-clockwise oriented  longitude of the circle component.}
\end{center}
\end{figure}
A presentation of the fundamental group of the Whitehead link can be computed to be $$\pi_1(L_{5^2_1})= \langle a, b, \lambda \vert \  b=\lambda a \lambda^{-1}, [\lambda^{-1},a^{-1}][\lambda^{-1},a][\lambda,a][\lambda,a^{-1}] = 1 \rangle$$

The $J(2,2n)$-twists knots, $n \in \mZ$, are obtained as $\frac{1}{n}$ Dehn filling along the circle component. The additional relation is thus $\mu^n=\lambda$, where $\mu = ba^{-1}b^{-1}a$. Notice that the second relation in the presentation above is $[\la,\mu]=1$, hence is redondant whence $\mu^n=\lambda$.
Figure \ref{twist} shows positive and negative twist knots, for $n=\pm 1$.
\begin{figure}[]
\begin{center}
\def\svgwidth{0.4\columnwidth}
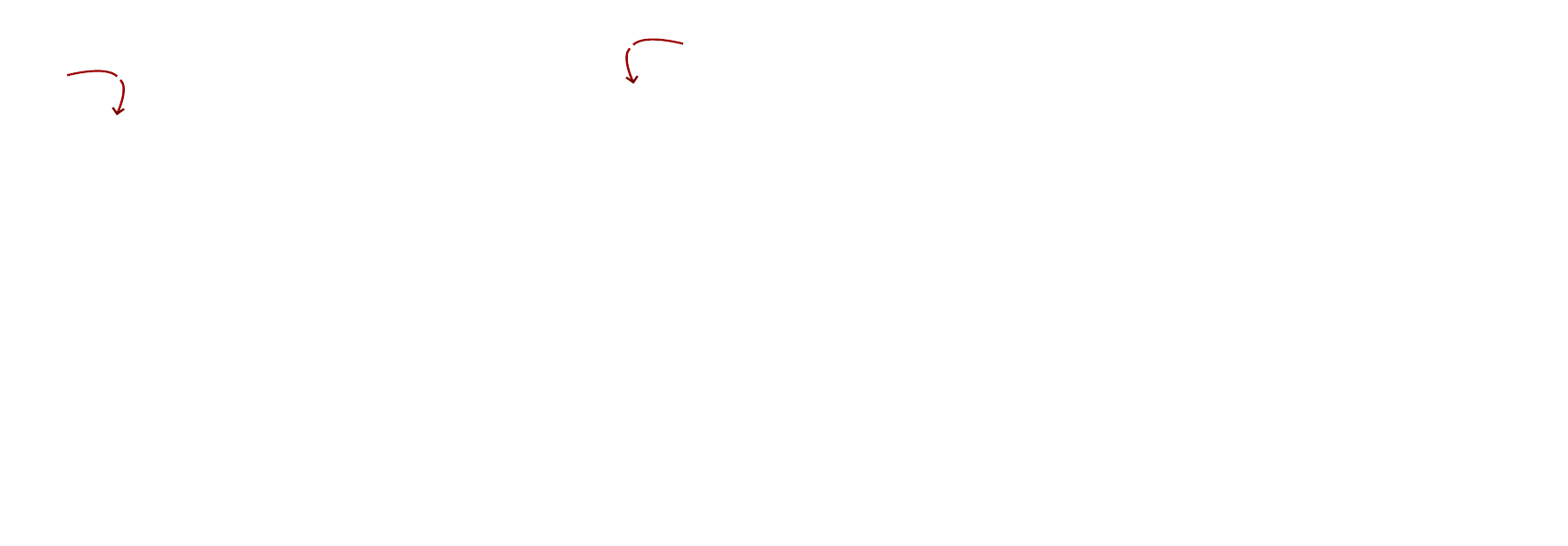
\caption{\label{twist}The diagram on the left is $J(2,2)$, the trefoil knot. The one on the right is $J(2,-2)$, the figure-eight knot.}
\end{center}
\end{figure}

We obtain the following presentation of twist knot group: $$\pi_1(J(2,2n)) = \langle a,b \vert (ba^{-1}b^{-1}a)^na=b(ba^{-1}b^{-1}a)^n \rangle=  \langle a,\lambda \vert \mu^n = \lambda \rangle$$ where the curve $\mu$ is the curve $ba^{-1}b^{-1}a=[\lambda,a][\lambda,a^{-1}]$.

We define  a tautological representation of the character variety $X(2,2n)$ of the twist knots $J(2,2n)$ by $$\rho(a)=\bm s & 1\\0& s^{-1} \ema, \rho(b) = \bm s & 0\\ y-s^2-s^{-2} & s^{-1} \ema$$ We will use the variable $x = s+s^{-1}$.
A direct computation shows (see \cite{Kit96}, for instance) that $\tor(M, \rho) = \frac{\det(\rho(\frac{\partial r}{\partial \lambda}))}{\det(1-\rho(a))}=\frac{\det(\rho(\frac{\partial r}{\partial \lambda}))}{(2-x)}$, where $\rho$ is extended linearly to the ring $\mZ[\pi_1]$.
\begin{itemize}
\item If $n >0$, we obtain $\tor(M, \rho) =\frac{ \det((1+\rho(\mu) +...+\rho(\mu)^{n-1})(1-\rho(b)+\rho(ba^{-1})-\rho(ba^{-1}b^{-1}))-1)}{2-x}$
\item If $n<0$, we obtain $\tor(M, \rho) = \frac{\det((\rho(\mu)^{-1}+...+\rho(\mu)^n) (1-\rho(b)+\rho(ba^{-1})-\rho(ba^{-1}b^{-1})))}{2-x}$
\end{itemize}
\subsubsection{The trefoil knot J(2,2)}
The character variety is the curve $X(2,2) =~\lbrace (x^2-y-2)(y-1) =0\rbrace$. The component of irreducible type $X$ is thus $\lbrace y-1 =0 \rbrace$. We compute the torsion function in $\mC[x,y]/(y-1)$, it is $\tor(M, \rho)=\frac{y-2x+3}{2-x} = 2$, the torsion is constant.
\subsubsection{The figure-eight knot J(2,-2)}
Let $X = \lbrace 2x^2+y^2-x^2y-y-1 =0 \rbrace$ the component of irreducible type of $X(2,-2)$. We have $\tor(M, \rho)=(4x^2-x^2y+y^2-y-6x+3)/(2-x) = 2x-2$ in $\mC[X]$, hence there is a zero at the point $\lbrace x=1,y=1 \rbrace$, with multiplicity 2.
\subsubsection{The knot $5_2$: J(2,4)}
Here $X = \lbrace -x^2(y-1)(y-2)+y^3-y^2-2y+1=0 \rbrace$, and $\tor(M, \rho)$ has two double zeros when $x=y$ are roots of $x^2-3x+1$.
\subsubsection{The knot $6_1$: J(2,-4)}
Here $X = \lbrace x^4(y-2)^2-x^2(y+1)(y-2)(2y-3)+(y^3-3y-1)(y-1)=0 \rbrace$, and $\tor(M, \rho)$ has three double zeros when $x=y$ are roots of $x^3-4x^2+3x+1$.

\begin{remk}
We observe that each time we have found a zero for the torsion, it had multiplicity 2 and $\lbrace \tor(M, \rho)=0\rbrace \subset X\cap\lbrace x=y \rbrace$. We have checked that this inclusion is strict.
\end{remk}

\section{The torsion at ideal points}
\label{sec:Ideal}
The aim of this section is to prove Theorem \ref{acyclic ideal}. The main idea to compute the order of the torsion $\tor(M, \rho)$ at an ideal point is to use an incompressible surface obtained at this ideal point by the Culler--Shalen theory. It will provide a Mayer--Vietoris splitting of the cohomological complex $C^*(M, \rho)$. By construction, despite the tautological representation $\rho$ is not convergent at an ideal point $v$ (see Proposition \ref{prop:Convergent}), this splitting will induce a convergent restriction of $\rho$ to the fundamental group of each connected component of the Culler--Shalen surface on one hand, and  convergent restrictions of $\rho$ to each connected component of the complement of this surface in $M$.

Depending whether the surface is separating in the manifold $M$ or not, the technicalities differ slightly : in Subsection \ref{split case} we assume that the surface is separating $M$, and in Subsection \ref{non-split case} we deal with the non-separating case.
\medbreak

We fix an ideal point $x \in \h{X}$.
Recall from Section \ref{tree} that such an ideal point provides an incompressible surface $\Si \in M$. Since $\partial \Sigma \subset \partial M$, each connected component induces the same element $[\partial \Sigma]$ in $\pi_1(M)$, well defined up to conjugacy. It is called the boundary slope of $\Sigma$.

In this section we will make the following assumptions on $\Si$, see Subsection \ref{prime}:
\begin{enumerate}[a)]
\item \label{hypo:parallel}
The surface $\Si$ is a union of homeomorphic parallel copies $\Si_1 \cup \ldots \cup \Si_n$, with $\chi(\Sigma_i) <0$.
\item \label{hypo:free}
Any connected component of the complement of any  $\Si_i$ in $M$ is a handlebody.
\item
\label{hypo:NonCentral}
There is an element $\gamma$ of $\pi_1(\Sigma)$ such that $\Tr \brho_\Sigma(\gamma) \neq 2$, so that $\brho_\Sigma$ is non-trivial.

Note that it does not depend on the restriction $\rho_{1,\Sigma}$ or $\rho_{2,\Sigma}$.
\item
\label{hypo:TrivialRoot}
The boundary of $\Sigma$ is not empty, and the trace of the image of the boundary slope $\partial \Sigma$ by the residual representation satisfies $\Tr \brho_\Sigma(\partial \Si)=2$.
\end{enumerate}

\begin{remk} \label{remk:Assumptiond}
When the surface $\Sigma$ is not separating in $M$ (the complement of any of its connected component is connected), then Assumption (\ref{hypo:TrivialRoot}) automatically holds. Indeed, since $M$ has rational homology of a circle, $H_2(M, \mathbb{Q})$ is trivial hence the surface $\Sigma$ needs to have a non-empty boundary in this case. In particular any component $\Sigma_i$ of $\Sigma$ provides a generator of the relative homology group $H_2(M, \partial M; \mathbb{Q})$. As the intersection pairing is non-degenerate, it implies that any connected component $\gamma \in \partial \Sigma$ is null-homologous in $M$. But $\Sigma$ is incompressible, hence $\gamma$ is a product of commutators in $\pi_1(\Sigma)$. Finally, we proved in Lemma \ref{lem:NonSplit} that the residual representation $\brho_\Sigma$ is reducible, hence any commutator has trace 2 by Lemma \ref{lem:Commutators}.
\end{remk}

\begin{remk} \label{eight knot}
An example of a separating surface satisfying Hypotheses  (\ref{hypo:parallel}), (\ref{hypo:free}), (\ref{hypo:NonCentral}), and (\ref{hypo:TrivialRoot}) is given by the figure-eight knot complement.
A classical diagram of the figure-eight knot is drawn in Figure \ref{huit}, with a non-orientable surface $\check{\Si}$ obtained by a checkerboard coloring. The boundary of a neighborhood of $\check{\Si}$ is an orientable surface $\Si$, which turns out to be incompressible. It is detected by the point $\lbrace x = \infty, y=2 \rbrace$ of the component of irreducible type of the character variety of the figure-eight knot. It easy to see on the picture that its complement is the union of two genus 2 handlebodies, and a computation shows that the trace of the boundary slope  $\partial \Si = uv^{-1}u^{-1}vuv^{-1}u^{-2}v^{-1}uvu^{-1}v^{-1}u^{-1}$  is $\Tr \brho_\Sigma(\partial \Sigma) = 2$.
\end{remk}

\medbreak

Now we can give a complete statement of Theorem \ref{acyclic ideal} from the Introduction. Its proof will occupy the rest of this section.

\begin{theo}
\label{theo:Separating}
Let $x \in \h{X}$ be an ideal point in the smooth projective model of $X$, and assume that an associated incompressible surface $\Sigma$ satisfies hypothesis (\ref{hypo:parallel}), (\ref{hypo:free}), (\ref{hypo:NonCentral}), and (\ref{hypo:TrivialRoot}).
Then the torsion function $\tor(M, \rho)$ has a pole at $x$. In particular the torsion function is non-constant.
\end{theo}

\subsection{The separating case} \label{split case}
In this section we give a proof of Theorem \ref{theo:Separating} in the separating case.

\medbreak


\begin{figure}[]
\begin{center}
\def\svgwidth{0.3\columnwidth}
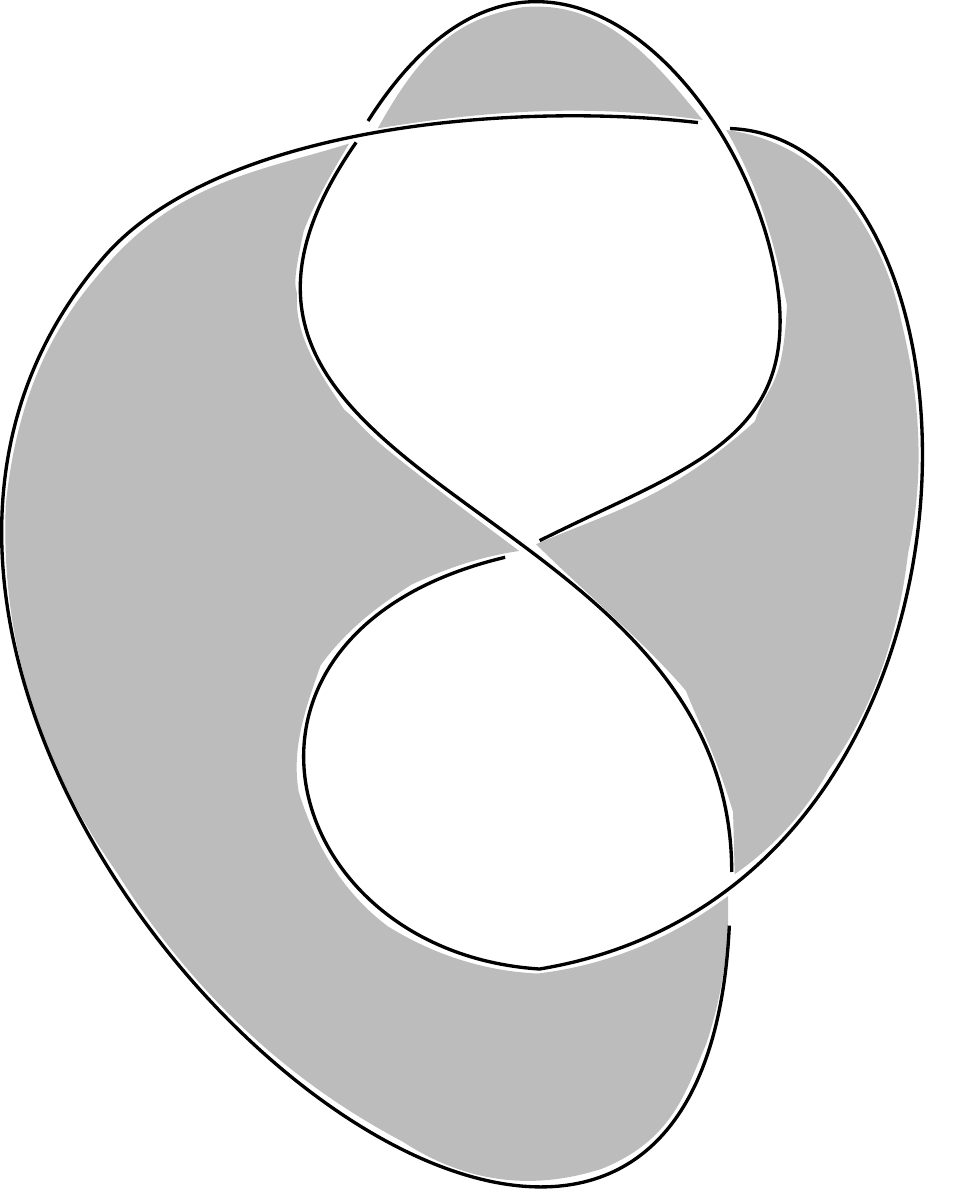
\caption{\label{huit}The figure-eight knot, with a non-orientable checkerboard surface $\check{\Si}$, and generating loops $u,v$ of its fundamental group}
\end{center}
\end{figure}

In this section we consider the case when the complement of any $\Sigma_i$ in $M$ is not connected. Thus we consider the splitting $M=M_1 \cup_{\Sigma_1} M_2$. It induces the following exact sequence of complexes of $k(X)$-vector spaces 
\begin{equation}
\label{equa:exact}
0 \to C^*(M, \rho) \to C^*(M_1, \rho_1)\oplus C^*(M_2, \rho_2) \to C^*(\Sigma, \rho_{1,\Sigma}) \to 0
\end{equation}

\begin{nota}
Since we have picked a base point $p \in \Si_1$ in Section \ref{prime}, we will abuse of the notation $\pi_1(\Si)$ to designate $\pi_1(\Si_1)$.
In the same way, we denote by $C^*(\Si, \rho_\Si)$ the twisted cohomological complex of $\Si_1$. Note that we make the choice $\rho_{1,\Sigma}$ here and in the sequel. Up to now, we wont mention it explicitly since it does not affect the cohomology of the complex $C^*(\Sigma, \rho_\Sigma)$, because $\rho_{1,\Sigma}$ and $\rho_{2,\Sigma}$ are conjugated. 
\end{nota}

\begin{lem} \label{333}
The long exact sequence in cohomology induced by (\ref{equa:exact}) reduces to the following isomorphism of $k(X)$-vector spaces: 
\begin{equation}
\label{eq:Isom}
H^1(M_1,\rho_1) \oplus H^1(M_2, \rho_2) \xrightarrow{\sim} H^1(\Sigma, \rho_\Sigma)
\end{equation}
\end{lem}
\begin{proof}
Recall that  $C^*(M, \rho)$ is acyclic and $H^j(\Sigma, \rho_\Sigma)=\lbrace 0 \rbrace$ for any $j \ge 2, i=1,2$ because $\Sigma$ have the same homotopy type that a one-dimensional CW complex. By Remark \ref{remk:SigmaFaithfull} 
there exists a $\g$ in  $\pi_1(\Sigma)$ such that $\Tr \rho_\Sigma(\g) \neq 2$. Thus $H^0(\Sigma, \rho_\Sigma) =~\lbrace 0 \rbrace$. The lemma follows now from the Mayer--Vietoris Theorem applied to (\ref{equa:exact}).
\end{proof}
To compute the torsion $\tor(M, \rho)$ in terms of the decomposition (\ref{equa:exact}), Lemma \ref{333} and the multiplicativity formula delayed in Proposition \ref{prop:Multiplicativity} indicate that we will have to focus on the isomorphism (\ref{eq:Isom}). Given $v$ the $k$-valuation on $k(X)$ corresponding to the ideal point $x$ in $\hat{X}$, once for all we fix a tautological representation as in Lemma \ref{converge}. In particular $\rho'_2$ is convergent although $\rho_2$ is not. Hence we replace (\ref{eq:Isom}) by the corresponding isomorphism involving $\rho'_2$:
\begin{lem}
\label{lem:Prime}
The isomorphism (\ref{eq:Isom}) can be replaced by 
\begin{align*}
\theta\colon H^1(M_1,\rho_1) \oplus &H^1(M_2, \rho'_2) \xrightarrow{\sim} H^1(\Sigma, \rho_\Sigma)\\
(Z_1,&Z_2)\mapsto(Z_1-U_n Z_2)_{\vert_\Sigma}
\end{align*}
\end{lem}
\begin{proof}
Recall the notation $U_n = \bsm t^n&0\\0&1\esm$, where $t \in k(X)$ is a uniformizing parameter for the valuation $v$.
By Lemma \ref{converge}, the convergent representation $\rho'_2$ is given by $\rho'_2(\gamma) = U_n^{-1}\rho_2(\gamma) U_n$ for any $\gamma$ in $\pi_1(M_2)$. Consequently there is an isomorphism 
\begin{align*}
H^1(M_2, \rho'_2) &\xrightarrow{\sim} H^1(M_2, \rho_2)\\
Z&\mapsto U_n Z
\end{align*}
and the lemma follows now from classical the Mayer--Vietoris sequence.
\end{proof}
We pick bases  $c_\Si, c_1,c_2,h_\Si, h_1, h_2$ of the complexes of $k(X)$-vector spaces $C^*(M_1, \rho_1)$, $C^*(M_2, \rho'_2)$ and $C^*(\Si, \rho_\Si)$ and of their homology groups $H^1(M_1, \rho_1)$, $H^1(M_2, \rho'_2)$ and $H^1(\Si, \rho_\Si)$. Note that each basis is a family of basis vectors for each $k(X)$-vector space. We also pick a basis for the acyclic complex $C^*(M, \rho)$. We will make use from the well-known "multiplicativity formula" due to Milnor \cite{Mil66}: 
\begin{prop}
\label{prop:Multiplicativity}
The torsion of the complex $C^*(M, \rho)$ can be expressed as 
$$\tor(M, \rho) = \frac{\tor(M_1, h_1) \tor(M_2, h_2)}{\tor(\Sigma, h_\Sigma)} \tor(\mathcal{H},h_1, h_2, h_\Sigma) \in k(X)^*$$
where $\mathcal{H}$ is the Mayer--Vietoris sequence induced by (\ref{equa:exact}).
 \end{prop}

Taking benefit of the convergence of the restricted tautological representation $\rho$ to $\pi_1(\Sigma)$ and to $\pi_1(M_i)$ (see Lemma \ref{converge}),  we consider the cohomological complexes of $\cO_v$-modules $C^*(\Sigma, \rho_\Sigma)_v$, $C^*(M_1, \rho_1)_v$ and $C^*(M_2, \rho'_2)_v$.
\begin{lem} \label{lem free}
The $\cO_v$-modules $H^1(M_1, \rho_1)_v, H^1(M_2, \rho'_2)_v, H^1(\Sigma, \rho_\Sigma)_v$ are free modules of rank $-\chi(\Si_1)$, $ -\chi(\Si_1)$ and $-2\chi(\Si_1)$ respectively. The rest of the cohomology is trivial.
\end{lem}
\begin{proof}
Let $\Theta$ denotes either $\Sigma$, $M_1$ or $M_2$.
The same argument than in the proof of Lemma \ref{333} implies that $H^i(\Theta, \rho_\Theta)_v$ is trivial for $i\ge2$.
 Since $H^0(\Theta, \rho_\Theta)_v$ is free by definition, and since its tensor product by $k(X)$ is trivial by Lemma \ref{333}, we conclude that $H^0(\Theta, \rho_\Theta)_v = \lbrace 0 \rbrace$. Hence we consider now the cohomology in degree~1.

By hypothesis (\ref{hypo:NonCentral}) and $\pi_1$-injectivity, there is a $\gamma$ in $\pi_1(\Theta)$ with $\Tr \brho_\Theta(\gamma) \neq 2$, hence the residual $k$-vector space $H^0(\Theta, \brho_\Theta)$ is trivial. The universal coefficients theorem implies that $H_0(\Theta, \brho_\Theta)$ also vanishes. Again by the universal coefficients theorem, the latter is isomorphic to $H_0(\Theta, \rho_\Theta)_v \otimes_{\cO_v} \cO_v/(t)$, so we have proved that the $\cO_v$-module $H_0(\Theta, \rho_\Theta)_v$ is torsion-free. Applying the universal coefficients theorem a third time, the (trivial) torsion part of the latter arises as the torsion part of $H_1(\Theta, \rho_\Theta)_v$, hence we obtain that $H^1(\Theta, \rho_\Theta)_v$ is a free $\cO_v$-module, as claimed. The computation of the rank follows from an Euler characteristic argument.

\end{proof}

\begin{prop}
The terms $\tor(M_1, h_1), \tor(M_2,h_2)$ and $\tor(\Si,h_\Sigma)$ lie in $\cO_v^*$
\end{prop}
\begin{proof}
Those factors are the torsions of based complexes of $k(X)$-vector spaces with based homologies, hence they lie in $k(X)^*$ by definition. Since the representations $\rho_1$, $\rho_2'$ and $\rho_\Si$ are convergent, one can define the complexes of $\cO_v$-modules $C^*(\Si, \rho_\Si)_v$, $C^*(M_1, \rho_1)_v$ and $C^*(M_2, \rho'_2)_v$ with their homology groups. Moreover, we might have chosen the bases $c_\Si, c_1,c_2,h_\Si, h_1, h_2$ of the paragraph above so that they generate those terms as $\cO_v$-modules, because $C^*(\Si, \rho_\Si)_v, \ldots , H^*(M_i, \rho_i)_v$ are free $\cO_v$-modules  (Lemma \ref{lem free}), and those choices would not affect the computation of the torsion of $M$ by Proposition \ref{prop:Multiplicativity}.

To be precise, assume that we have chosen a basis $h_2^1 = \lbrace h_2^{1,1}, \ldots, h_2^{1,\dim H^1(M_2, \rho'_2)}\rbrace$ of the free $\cO_v$-module $H^1(M_2, \rho'_2)_v$ that spans $H^1(M_2, \rho'_2)$ as a $k(X)$-vector space, and that it is mapped on a basis through the isomorphism of $k(X)$-vector spaces $H^1(M_2, \rho'_2) \to H^1(M_2, \rho_2)$.  Finally, the map $H^1(M_1, \rho_1)_v \to H^1(\Si, \rho_\Si)_v$ identifies the basis $h_1^1$ to a sub-basis of $h_{\Si}^1$ as an $\cO_v$-module, which can be completed in a basis of $H^1(\Si, \rho_\Si)_v$. Now we prove that the torsions of this complexes lie in $\cO_v^*$. We give the argument for $M_1$, but it can be shown in general in the same way.
The complex is $C^0(M_1, \rho_1)_v \xrightarrow{A} C^1(M_1, \rho_1)_v$. Since $H^0(M_1, \rho_1)_v$ is trivial, the matrix $A$ is the matrix of an injective $\cO_v$-linear morphism. Moreover, $H^1(M, \rho_1)_v$ is free, hence its determinant is an invertible $\det A \in \cO_v^*$ as claimed.
\end{proof}
\begin{remk}
As a consequence of this proposition, it is enough to compute the term $\tor(\mathcal{H},h^1_1, h^1_2, h^1_\Sigma)$. By Proposition \ref{prop:Cayley} and Lemma \ref{lem:Prime} it is the inverse of the determinant of the following map:
\begin{align*}
\theta\colon H^1(M_1,\rho_1) \oplus &H^1(M_2, \rho'_2) \xrightarrow{\sim} H^1(\Sigma, \rho_\Sigma)\\
(Z_1,&Z_2)\mapsto(Z_1-U_n Z_2)_{\vert_\Sigma}
\end{align*}
We compute now $\det(\theta)$.
\end{remk}
For this purpose we observe that the relation $\rho_1(\g) = U_n \rho'_2(\g) U_n^{-1}$ holds for any $\g$ in  $\pi_1(\Sigma)$. It implies that the corresponding residual representations have the following form: for any $\gamma$ in $\pi_1(\Sigma)$:
 \begin{equation}
 \label{eq:ResidualReducible}
 \brho_1(\g) = \bm \la(\g) & 0\\ *  & \la^{-1}(\g) \ema, \brho_2(\g) = \bm \la(\g) & *\\ 0 & \la^{-1}(\g) \ema.
 \end{equation}
\begin{nota}
Recall that we denote by $C^*(\Sigma, \lambda)$ the cohomological complex of $\Sigma$ with coefficient in $k$ with the action of $\pi_1(\Sigma)$ induced by multiplication by $\lambda$.
\end{nota}
%
%
%
\begin{lem} 
\label{lem:LongSequence}
Denote $Z_i$ in $C^1(\Sigma, \brho_i)$ by $Z_i(\g) = \bsm x_i(\g) \\ y_i(\g) \esm$. There are exact sequences of $\mZ[\pi_1(\Sigma)]$-modules:
\begin{align*}
0 \to C^1(\Sigma, \lambda^{-1}) &\xrightarrow{i_1} C^1(\Sigma, \brho_{1,\Sigma})\xrightarrow{p_1} C^1(\Sigma, \lambda)\to 0\\
y_1 &\mapsto \bsm0\\y_1 \esm \ \ , \ \bsm x_1\\y_1 \esm \mapsto x_1\\
\end{align*}
and
\begin{align*}
0 \to C^1(\Sigma, \lambda) &\xrightarrow{i_2} C^1(\Sigma, \brho_{2,\Sigma})\xrightarrow{p_2} C^1(\Sigma, \lambda^{-1})\to 0\\
x_2 &\mapsto \bsm x_2\\0 \esm \ \ , \ \bsm x_2\\ y_2 \esm \mapsto y_2\\
\end{align*}
Moreover, it induces the following exact sequences:
$$0 \to H^1(\Sigma, \la^{-1}) \xrightarrow{i_1} H^1(\Sigma, \brho_{1,\Si}) \xrightarrow{p_1} H^1(\Sigma, \la) \to 0$$
$$0 \to H^1(\Sigma, \la) \xrightarrow{i_2} H^1(\Sigma, \brho_{2,\Si}) \xrightarrow{p_2} H^1(\Sigma, \la^{-1}) \to 0$$
\end{lem}
\begin{proof}
The fact that the sequences of complexes are exact is obvious, that the maps are morphisms of $\mZ[\pi_1(\Sigma)]$-modules follows directly from the explicit description (\ref{eq:ResidualReducible}) of the action of $\brho_{i,\Sigma}$ on $k^2$, for $i=1,2$.

The exactness of the sequences in cohomology follows from the long exact sequence in cohomology and because hypothesis (\ref{hypo:NonCentral}) implies the nullity of the vector spaces $H^i(\Sigma, \lambda^{\pm1})$ for $i \neq 1$.
\end{proof}

We want to prove that the torsion has a pole at the ideal point $x$ or equivalently,  that the determinant of the map 
\begin{align*}
\theta\colon H^1(M_1,\rho_1) &\oplus H^1(M_2, \rho'_2) \xrightarrow{\sim} H^1(\Sigma, \rho_\Sigma)\\
(Z_1&,Z_2)\mapsto(Z_1-U_n Z_2)_{\vert_\Sigma}
\end{align*}
has positive valuation.
See $\theta_v$ as a morphism of $\cO_v$-modules $$\theta\colon H^1(M_1,\rho_1)_v \oplus H^1(M_2, \rho'_2) \to H^1(\Sigma, \rho_\Sigma)_v$$ and consider the $k$-linear residual map $\bar{\theta}\colon H^1(M_1, \brho_1) \oplus H^1(M_2, \brho_2) \to H^1(\Si, \brho_\Si)$ obtained by reducing $\theta_v$ modulo $(t)$. It maps $(\overline{Z}_1, \overline{Z}_2)$ onto ${\overline{Z}_1}_{\vert_\Sigma} - \bsm 0 \\ {\overline{y}_2}_{\vert_\Sigma} \esm$ where the bar denotes the reduction mod $(t)$.
\begin{lem} \label{det0}
The torsion has a pole at $x$ if and only if $\bar{\theta}$ is not an isomorphism.
\end{lem}
\begin{proof}
It is clear from the fact that $\det(\bar{\theta}) = (\det{\theta})(0)$, that is $v(\det(\theta)) \ge 0$ if and only if $\det(\bar{\theta}) = 0$. 
\end{proof}

Let us prove that $\overline{\theta}$ is not an isomorphism. Let $\partial M_2$ be the boundary of $M_2$, and $\partial \Sigma \subset \partial M_2$ be the union of the boundary components of $\Si$.  
\begin{lem}
\label{lem:Boundaryslope}
Let  $\brho_{2,\partial \Sigma}$ be the restriction of $\brho_{2,\Sigma}$ to the boundary  $\partial \Sigma$ of $\Sigma$. Then the $k$-vector space $H^1(\partial \Sigma, \brho_{2,\partial \Sigma})$ is non-trivial.
\end{lem}
\begin{proof}
By hypothesis (\ref{hypo:TrivialRoot}) we have that $\brho_2([\partial \Sigma]) = \bsm 1 & *\\ 0 & 1 \esm$, hence $H^0(\partial \Sigma, \brho_{2,\partial \Sigma}) $ is not trivial, and the lemma follows by Poincar\'e duality.
\end{proof}

The same long exact sequence for the coefficients of $C^*(\partial \Sigma, \brho_{2,\partial \Sigma})$ than in Lemma \ref{lem:LongSequence} ends up with 
$$\ldots \to H^1(\partial \Sigma, \brho_{2,\partial \Sigma}) \xrightarrow{p_{\partial \Sigma}} H^1(\partial \Sigma,k) \to 0$$
Since $H^1(\partial \Sigma,k)$ has dimension the number of connected components of $\Sigma$, in particular it is not trivial, and the map $p_{\partial \Sigma}\colon H^1(\partial \Sigma, \brho_{2,\partial \Sigma}) \to H^1(\partial \Sigma, k)$ is not zero.
On the other hand, the inclusion $\Si \subset \partial M_2$ provides the sequence
\begin{equation} \label{M2 bord}
H^1(\partial M_2, \Si; \brho_{2,\partial M_2}) \to H^1( \partial M_2, \brho_{2,\partial M_2}) \to H^1(\Si,  \brho_{2,\Si}) \to H^2(\partial M_2, \Si; \brho_{2,\partial M_2}) \to 0
\end{equation}
Denote by $A$ the union of small annulus neighborhood of the components of $\partial \Sigma$ in $\partial M_2$. By excision, we have $H^2(\partial M, \Si; \brho_{2,\partial M_2}) \simeq H^2(A, \partial A; \brho_{2,\partial \Sigma})$. Now Poincar\'e-Lefschetz duality implies that the latter is isomorphic to $H_0(A, \brho_{2,\partial \Sigma})$. Now $A$ retracts on $\partial \Sigma$, hence we obtain the isomorphism is $H^2(\partial M, \Si; \brho_{2,\partial M_2})\simeq H_0(\partial \Sigma, \brho_{2,\partial \Sigma})$. Again by duality we obtain $H^2(\partial M, \Si; \brho_{2,\partial M_2})\simeq H^1(\partial \Sigma, \brho_{2,\partial \Sigma})$.

We summarize that in the following commutative diagram:
\begin{center}
\begin{tikzpicture}
\node (A) {$H^1(M_2, \brho_2)$};
\node (B) [below of=A] {$H^1(\partial M_2, \brho_{2, \partial M_2})$};
\node (C) [below of=B] {$H^1(\Si, \brho_{2, \Si})$};
\node (D) [left of=C, xshift=-2cm] {$H^1(\Si, \la)$};
\node (E) [right of=C, xshift=2cm] {$H^1(\Si, \la^{-1})$};
\node (F) [below of=C] {$H^1(\partial \Sigma, \brho_{2, \partial \Sigma}))$};
\node (G) [right of=F, xshift=2cm] {$H^1(\partial \Sigma, k)$};
\node (H) [right of=E] {$0$};
\node (I) [right of=G] {$0$};
\node (J) [below of=G] {$0$};
\draw[->] (A) to node {$i_{\partial M_2}$} (B);
\draw[->] (B) to node {$i_{\Si}$} (C);
\draw[->] (C) to node {$p_2$} (E);
\draw[->] (D) to node {$i_2$} (C);
\draw[->] (C) to node {$i_{\partial \Si}$} (F);
\draw[->] (F) to node {$p_{\partial \Sigma}$} (G);
\draw[->] (E) to node {$i_{\partial \Si}$} (G);
\draw[->] (E) to (H);
\draw[->] (G) to (I);
\draw[->] (G) to (J);
\draw[->, dotted] (A) to node {$F$} (E);
\end{tikzpicture}
\end{center}

\begin{lem} \label{lowdim}
The composition map $$F: H^1(M_2, \brho_2) \xrightarrow{i_{\partial M_2}} H^1(\partial M_2, \brho_{2, \partial M_2}) \xrightarrow{i_\Si} H^1(\Si, \brho_{2, \Si}) \xrightarrow{p_2} H^1(\Si, \la^{-1})$$ is not an isomorphism.
\end{lem}
\begin{proof}
The first observation is that $\dim H^1(M_2, \brho_2) = \frac{\dim H^1(\Si, \brho_{2, \Si})}{2} = \dim H^1(\Si, \la^{-1})$.
We need to prove that the map $F$  is not onto. By way of contradiction, assume that the map $F$ is onto, hence it has a non-trivial image in $H^1(\partial \Sigma, k)$ through the map $i_{\partial \Si}$. On the other hand, the vertical sequence $H^1(\partial M_2,  \brho_{2, \partial M_2}) \xrightarrow{i_{\partial M_2}} H^1(\Si, \brho_{2, \Si}) \xrightarrow{i_{\partial \Si}} H^1(\partial \Sigma, \brho_{2, \partial \Sigma})$ is exact by equation (\ref{M2 bord}), and the commutativity of the diagram shows that $i_{\partial \Si} \circ F =0$, a contradiction. It proves the lemma.
\end{proof}

\begin{proof}[Proof of Theorem \ref{theo:Separating}]
We just have to observe that $\bar{\theta}\colon H^1(M_1, \brho_1) \oplus H^1(M_2, \brho_2) \to H^1(\Si, \brho_{1, \Si})$ is the direct sum of \benu
\item the injective map $H^1(M_1, \brho_1) \to H^1(\Si, \brho_{1, \Si})$ induced by inclusion 
\item the map $H^1(M_2, \brho_2) \to H^1(\Si, \brho_{2, \Si}) \xrightarrow{p_2} H^1(\Si, \la^{-1}) \xrightarrow{i_1} H^1(\Si, \brho_{1, \Si})$ which is $i_1 \circ F$.
\eenu
The first map has maximal rank $-\chi(\Si)$, but the second has rank smaller than $-\chi(\Si)$ by Lemma \ref{lowdim}. Hence $\bar{\theta}$ is not onto and by Lemma \ref{det0} we conclude that the torsion vanishes at $x$. It proves the theorem.
\end{proof}

\subsection{The non-split case} \label{non-split case}
In this section we prove Theorem \ref{theo:Separating} in the case where the complement of any connected component of the incompressible surface associated to $x$ is connected. 

\medbreak

Recall that we have $\Sigma=\Sigma_1 \cup \ldots \cup \Sigma_n$ union of $n$ parallel connected copies. We fix a base-point $p \in \Sigma_1$, and that we identify $\pi_1(\Sigma)$ with $\pi_1(\Sigma_1)$. We have the following splitting 
\begin{equation}
\label{equa:splitting}
M =~H \cup_{\Sigma_1 \cup \Sigma_n} V(\Sigma)
\end{equation}
where $V(\Sigma)$ is a neighborhood of $\Sigma$ homeomorphic to $\Sigma_1 \times [0,1]$ with $\partial V(\Sigma) = \Sigma_1 \cup \Sigma_n$. We identify as well $\pi_1(V(\Sigma))$ with $\pi_1(\Sigma)$. Given $\al\colon \pi_1(\Sigma_1) \to \pi_1(\Sigma_n)$, one can write the fundamental group of $M$ as the following HNN group extension $\pi_1(M) =~\langle \pi_1(H), v \  \vert \ v \g v^{-1} =~\al(\g), \forall \g \in \pi_1(\Sigma) \rangle$.

We denote by $\rho_1\colon \pi_1(\Sigma) \to \Sldeux(\cO_v)$ the restriction of $\rho$ to $\pi_1(\Sigma)$, and similarly by $\rho_n\colon \pi_1(\Sigma_n) \to~\Sldeux(k(X))$ its restriction to $\pi_1(\Sigma_n) = v \pi_1(\Sigma) v^{-1}$, hence $\rho_n (\g) = V_n \rho_1(\g) V_n^{-1}= \check{U}_n^{-1} \rho_1(\gamma) \check{U}_n$ for $ \g \in \pi_1(\Sigma)$ (see Lemma \ref{lem:NonSplit} for notations).
The decomposition (\ref{equa:splitting}) induces the following exact sequence of twisted complexes:
\begin{equation*}
0 \to C^*(M, \rho) \to C^*(H, \rho_H) \oplus C^*(\Sigma, \rho_1) \to C^*(\Sigma, \rho_1) \oplus C^*(\Sigma, \rho_n) \to 0
\end{equation*}
The following proposition recaps the series of lemmas in Subsection \ref{split case}, we refer to the corresponding lemmas for proofs, that translate in exactly the same way here. We use the isomorphism $H^1(\Sigma, \rho_n) \to H^1(\Sigma, \rho_1)$, $Z \mapsto \check{U}_n Z$.
\begin{prop}
\label{prop:pole}
The vanishing order of the torsion at the ideal point $x \in \bar{X}$ is given by $-v(\det \theta)$, where the isomorphism $\theta$ is given by 
\begin{align*}
\theta\colon H^1(H, \rho_H) &\oplus H^1(\Sigma, \rho_1) \to H^1(\Sigma, \rho_1) \oplus H^1(\Sigma, \rho_1)\\
(Z_1&, Z_2) \mapsto ({Z_1}_{\vert\Sigma}-Z_2, \check{U}_n ({Z_1}_{\vert\Sigma}-Z_2))
\end{align*}
\end{prop}

As usual we will denote by $\brho_1 \colon \pi_1(\Sigma) \to \SL_2(\cO_v) \to \SL_2(k)$ the residual representation induced by $\rho_1$.
We want to show that $v(\det \theta) > 0$. We focus on the residual map, which has the following explicit expression: 
\begin{align*}
\bar{\theta}\colon H^1(H, \brho_H) &\oplus H^1(\Sigma, \brho_1) \to H^1(\Sigma, \brho_1) \oplus H^1(\Sigma, \brho_1) \\
(\overline{Z}_1&,\overline{Z}_2) \mapsto \left(\overline{Z}_1\vert_\Sigma-\overline{Z}_2\vert_\Sigma, \bm -\overline{x}_1\vert_\Sigma+\overline{x}_2\vert_\Sigma \\ 0 \ema\right)
\end{align*}
where $\overline{Z}_i = \bm \overline{x}_i \\ \overline{y}_i \ema$.
We show that it has a non-trivial kernel.

\begin{lem} \label{kernel}
The map 
\begin{align*}
\bar{\theta}_2\colon H^1(H, \brho_H)\oplus H^1(\Sigma, \brho_1) &\to H^1(\Sigma, \brho_1)\\
(\overline{Z}_1, \overline{Z}_2) &\mapsto \bm  -\overline{x}_1\vert_\Sigma+\overline{x}_2 \\ 0 \ema
\end{align*}
is not onto.
\end{lem}
\begin{proof}
We know from lemma \ref{lem:NonSplit} that the residual representation $\brho_1$ has the form $\brho_1(\g) = \bsm \la(\g) & 0 \\ * & \la^{-1} (\g) \esm$, for $\g \in \pi_1(\Sigma)$. Using the same arguments than in Subsection \ref{split case} we have the exact sequence 
\begin{align*}
0 \to H^1(\Sigma, \la) \to H^1(\Sigma&, \brho_1) \to H^1(\Sigma, \la^{-1}) \to 0\\
x \mapsto \bsm x\\0\esm , &\ \bsm x\\y\esm \mapsto y
\end{align*}
And we observe that the image of $\bar{\theta}_2$ is included in the strict subspace $H^1(\Sigma, \la)$ of  $H^1(\Sigma, \brho_1)$.
\end{proof}

\begin{proof}[Proof of Theorem \ref{theo:Separating}]
By Lemma \ref{kernel}, the map $\bar{\theta}_2$ is not onto, hence the map $\theta$ is not onto, in particular its determinant vanishes. Now we deduce directly from Proposition \ref{prop:pole} that the torsion has a pole at $x$, and it proves the theorem.
\end{proof}

\bibliography{biblio} 
\bibliographystyle{plain}

\end{document}